\documentclass[10pt]{article}

\usepackage{amssymb,amsmath,amsthm,mathrsfs,amsfonts,graphicx,epsf}
\usepackage{color}
\usepackage{xcolor}

\usepackage[hypertex]{hyperref} 

 \usepackage{esint}
 \usepackage{fullpage}

\newcommand{\Reg}{\mathsf{R}}

\newcommand{\Sing}{\mathsf{S}}

\newcommand{\ord}{\mathrm{ord}}

\newcommand{\cH}{{\cal H}}

\newcommand{\cI}{{\cal I}}
\newcommand{\cF}{{\cal F}}
\newcommand{\cP}{{\cal P}}
\newcommand{\codim}{\mathrm{codim}}

\newcommand{\bbrr}{B_{\widehat d_p}}

\newcommand{\Leb}[1]{{\mathscr L}^{#1}}

\newcommand{\volh}{\mathrm{vol}_H}
\newcommand{\vol}{\mathrm{vol}}

\newcommand{\bD}{{\cal D}}
\renewcommand{\wp}{\hat \varpi^p}
\newcommand{\mup}{\hat \mu^p}
\newcommand{\muq}{\hat \mu^q}

\newcommand{\reg}{\Reg}

\newcommand{\cL}{\mathcal{L}}

\newcommand{\cX}{{\cal X}}
\newcommand{\cY}{{\cal Y}}

\newcommand{\ee}{{e}}

\newcommand{\ps}{\langle\cdot,\cdot\rangle}
\renewcommand{\span}{\mathrm{span}}

\newtheorem{theorem}{Theorem}[section]
\newtheorem{corollary}[theorem]{Corollary}
\newtheorem{lemma}[theorem]{Lemma}
\newtheorem{proposition}[theorem]{Proposition}

\theoremstyle{definition}
\newtheorem{definition}[theorem]{Definition}

\theoremstyle{remark}
\newtheorem{remark}[theorem]{Remark}
\newtheorem{example}[theorem]{Example}
\newcommand{\bt}{\begin{theorem}}
\newcommand{\et}{\end{theorem}}
\newcommand{\bl}{\begin{lemma}}
\newcommand{\el}{\end{lemma}}
\newcommand{\bp}{\begin{proposition}}
\newcommand{\ep}{\end{proposition}}
\newcommand{\bc}{\begin{corollary}}
\newcommand{\ec}{\end{corollary}}
\newcommand{\bdeff}{\begin{definition}}
\newcommand{\edeff}{\end{definition}}
\newcommand{\brem}{\begin{remark}}
\newcommand{\erem}{\end{remark}}
\renewcommand{\r}[1]{(\ref{#1})}
\newcommand{\con}{{\mathcal C}}

\def\diam{\mathop{\mathrm{diam}}}
\def\ss{\mathcal{S}}
\newcommand{\hh}{{\mathcal H}}

\newcommand{\qn}{{Q_N}}

\newcommand{\bi}{\begin{itemize}}
\newcommand{\iii}{\item}
\newcommand{\ei}{\end{itemize}}
\newcommand{\bd}{\begin{description}}
\newcommand{\ed}{\end{description}}

\newcommand{\bqn}{\begin{eqnarray}}
\newcommand{\eqn}{\end{eqnarray}}
\newcommand{\eqnn}{\nonumber\end{eqnarray}}

\newcommand{\ba}[1]{\begin{array}{#1}}
\newcommand{\ea}{\end{array}}

\newcommand{\R}{\mathbb{R}}

\newcommand{\N}{\mathbb{N}}

\newcommand{\g}{\gamma}

\newcommand{\eps}{\epsilon}

\newcommand{\dSdm}{\frac{d\ss^{Q_\reg}\llcorner_\Reg}{d\mu}}
\newcommand{\VecM}{\mathrm{Vec}(M)}

\title{Hausdorff volume in non equiregular sub-Riemannian manifolds}

\author{
 R.~Ghezzi\thanks{Institut de Math\'ematiques de Bourgogne UBFC, 9 Avenue Alain Savary BP47870 21078 Dijon Cedex France     {\tt roberta.ghezzi@u-bourgogne.fr}},
   F.~Jean\thanks{828 Boulevard des Mar\'{e}chaux
91762 Palaiseau, France,  and
  Team  GECO, INRIA Saclay --
\^{I}le-de-France, {\tt frederic.jean@ensta-paristech.fr}}
\thanks{This work was partially supported by the European project AdG ERC ``GeMeThNES'', grant agreement number 246923  (see also {\tt gemethnes.sns.it}); by iCODE (Institute for Control and Decision),
research project of the IDEX Paris- Saclay; and by the Commission
of the European Communities under the 7th Framework Programme Marie
Curie Initial Training Network (FP7-PEOPLE-2010-ITN), project SADCO,
contract number 264735.}
}

\begin{document}

\maketitle

\begin{abstract}
In this paper we study the Hausdorff volume in a non equiregular sub-Riemannian manifold and we compare it with a smooth volume. We first give the Lebesgue decomposition of the Hausdorff volume. Then we study the regular part, show that it is not commensurable with the smooth volume, and give conditions under which it is a Radon measure. We finally give a complete characterization of the singular part. We illustrate our results and techniques on numerous examples and cases (e.g. to generic sub-Riemannian structures).
\end{abstract}

\tableofcontents

\section{Introduction}

The present work is motivated by the analysis of intrinsic volumes in sub-Riemannian geometry. Here   a sub-Riemannian manifold is a triplet
 $(M,\bD,g)$, where  $M$ is a
smooth manifold, $\bD$ a Lie-bracket generating distribution on $M$  and $g$ a   Riemannian metric on $\bD$ (note that our framework  will permit us to consider rank-varying distributions as well). As in Riemannian geometry,  one defines the length of
  absolutely continuous paths which are
almost everywhere tangent to $\bD$   by integrating the $g$-norm of
their tangent vectors.  Then, the sub-Riemannian
distance $d$ is defined as the infimum of length of paths between two given points. Since $\bD$ is Lie-bracket generating, for every point $p\in M$ there exists $r(p)\in\N$ such that
\begin{equation}\label{flagintro}
\{0\}=\bD^0_p\subset\bD_p^1\subset\dots\subset\bD_p^{r(p)}=T_pM,
\end{equation}
where $\bD^i_p=\{X(p)\mid X\in\bD^i\}$ and $\bD^i\subset\VecM$ is the submodule defined recursively  by
  $
\bD^{1}=\bD$, $\bD^{i+1}=\bD^{i}+[\bD,\bD^{i}]
$. The sub-Riemannian manifold is equiregular if the dimensions $\dim\bD^i_p$ do not depend on $p$.

%

Intrinsic measures on  sub-Riemannian manifolds are those which are associated with the sub-Riemannian structure.
There are essentially two ways to build such measures: either using the metric structure defined by the sub-Riemannian distance which provides Hausdorff and spherical Hausdorff measures, or by means of the algebraic structure associated with the distribution which allows to construct the so called Popp's measure (see \cite{montgomery}) on equiregular manifolds.

Intrinsic measures have been widely studied in the equiregular case, where the algebraic structure is well understood (in Carnot groups \cite{tyson,bonfi,FSSC} and on equiregular manifolds \cite{gz06}).
The relevance of the study of intrinsic volumes, e.g., top-dimensional Hausdorff measures, is due to their use in PDE's analysis. For instance, to generalize the Laplace-Beltrami operator in sub-Riemannian geometry, one needs a (sufficiently smooth) intrinsic volume: this motivates the analysis of regularity of the Hausdorff volume in   \cite{balu}.  We also mention the recent work \cite{capognaledonne} where smoothness of intrinsic volumes is needed to apply some nice PDE's proof ideas.
 In the non equiregular case, no study of intrinsic volumes exists so far.    We refer the reader to \cite{gro96} for a survey of many facts and interesting questions.   Hausdorff measures are also studied in \cite{noi,cocv_fj}  along curves.

The main aim of this paper is the analysis of the Hausdorff volume in non equiregular sub-Riemannian manifolds.

For convenience, let us first recall the equiregular case. Let $(M,\bD,g)$ be an equiregular sub-Riemannian manifold. The Hausdorff dimension of $M$ can be algebraically computed in terms  of the flag \r{flagintro} of the distribution
 by $\dim_HM=Q$, where
 $$
Q=\sum_{i=1}^{r}i(\dim\bD^i_p-\dim\bD^{i-1}_p),
 $$
see \cite{mitchell}. The Hausdorff volume, denoted by $\volh$, is by definition the top-dimensional spherical Hausdorff measure $\ss^{\dim_HM}$.  Assume $M$ to be oriented. A natural way to understand the behavior of $\volh$ is to compare it with a smooth volume $\mu$  on $M$, i.e., a measure defined on open sets by $\mu(A)=\int_A \omega$, where $\omega\in\Lambda^nM$ is a positively oriented non degenerate $n$-form. The equiregular assumption implies that $\volh$ and any smooth volume $\mu$ are mutually absolutely continuous. Moreover the Radon-Nikodym derivative of $\volh$ with respect to $\mu$ at a point $p$, denoted by $\frac{d\volh}{d\mu}(p)$,  can be computed explicitly by the formula (see \cite{balu})
\begin{equation}\label{densintro}
\lim_{r\to 0}\frac{\volh(B(p,r))}{\mu(B(p,r))}=\frac{2^Q}{\mup(\widehat B_p)},
\end{equation}
where $B(p,r)$ is the sub-Riemannian ball centered at $p$ of radius $r$, $\widehat B_p$ is the unit ball in the nilpotent approximation at $p$ and $\mup$ is a measure obtained through a blow-up procedure of $\mu$ at $p$. As a consequence, $\frac{d\volh}{d\mu}$ is continuous on $M$ and hence locally bounded and locally bounded away from zero on $M$. With the language of \cite{mitchell, montgomery}, this implies that $\volh$ and $\mu$ are commensurable and, in particular, that $\volh$ is a Radon measure, i.e., $\volh(K)<\infty$ for every compact set $K$.  Therefore, when the manifold is equiregular, $\volh$ is well understood in the sense that it behaves essentially as a smooth volume. Nevertheless, further regularity of $\frac{d\volh}{d\mu}$ is not granted see \cite{balu}.

In this paper we study the Hausdorff volume in  a non equiregular sub-Riemannian manifold. A point $p$ is called regular if the growth vector $q\mapsto(\dim\bD^1_q,\dots, \dim\bD^{r(q)}_q)$ is  constant in a neighborhood of $p$, otherwise $p$ is called singular. The natural assumption under which we perform our work is that the manifold is stratified by equisingular submanifolds, where  both the growth vector of the distribution and the growth vector of the distribution restricted to the submanifold are  constant. These submanifolds were introduced in \cite{gro96} and, thanks to their simple Lie algebraic structure, they constitute the fundamental block that allows us to carry through our investigation.

When the set $\Sing$ of  singular points is not empty, the Hausdorff dimension of $M$ is obviously the maximum between the Hausdorff dimension of $\Sing$ and the Hausdorff dimension of the set $\Reg=M\setminus\Sing$ of regular points.

The first question is whether $\volh$ is absolutely continuous with respect to a smooth volume $\mu$.
It turns out that this may not be the case and   $\volh$   may have a singular part.
More precisely, under the assumption that   $\Sing$   is $\mu$-negligible, then $\volh$ admits the following Lebesgue decomposition
$$
\volh=\volh\llcorner_{\Reg}+\volh\llcorner_{\Sing},
$$
that is, $\volh\llcorner_{\Reg}\ll\mu$ and  $\volh\llcorner_{\Sing}\perp\mu$  (see Corollary~\ref{jkl}). As a consequence, $\volh$ is absolutely continuous with respect to a smooth volume $\mu$ if and only if the Hausdorff dimension of the singular set is smaller than the Hausdorff dimension of the set of regular points.

The next problem is to analyze the behavior of the absolutely continuous part $\volh\llcorner_\Reg$ (which is nontrivial if and only if $\dim_H\Reg\geq\dim_H\Sing$). Thanks to \r{densintro}, this amounts to study the function $q\mapsto \muq(\widehat B_q)$ near singular points.
 To this aim, we show that the asymptotics of this function is characterized by the one of  the determinants of   adapted bases at regular points (see Proposition~\ref{we}). As a direct consequence, we deduce that  $\frac{d\volh\llcorner_\Reg}{d\mu}(q)$ blows up when  $q$ approaches the singular set   (see Proposition~\ref{bdd}).
 In particular,  $\frac{d\volh\llcorner_\Reg}{d\mu}$ is not essentially bounded near the singular set, that is,  unlike the equiregular case, $\mu$ and $\volh\llcorner_{\Reg}$ are no longer  commensurable.
Going further in the regularity analysis, we find out that  $\volh\llcorner_{\Reg}$ may even fail to be  locally integrable with respect to $\mu$ and therefore $\volh\llcorner_{\Reg}$ may fail to be a Radon measure.
More precisely, we exhibit a sufficient condition involving the nonholonomic order of $\frac{d\volh\llcorner_\Reg}{d\mu}$  and the algebraic structure of the distribution at a singular point for non-integrability of $\volh\llcorner_\Reg$ (see Proposition~\ref{fin}), and another one involving the usual order of functions and the codimension of the singular set (see Proposition~\ref{le:domega=0}). For instance,  these conditions are satisfied when the Hausdorff dimension of $\Reg$ is not greater  than   the Hausdorff dimension of $\Sing$ or when $\Sing$ topologically  splits $M$  (see Corollaries \ref{th:s>r} and \ref{tre}). We also exhibit a sufficient condition for the integrability of $\volh\llcorner_{\Reg}$ (see Proposition \ref{finito}), but there is a gap between integrability and non-integrability conditions.
In Figure~\ref{sumup}
 we summarize the relations between the Hausdorff volume and $\mu$ in all cases.

\begin{figure}[h!]
\begin{center}
\input{sumup_fred.pstex_t}

\caption{Summary of relations between $\volh$ and a smooth volume $\mu$ on an oriented stratified sub-Riemannian manifold ($Q_\Sing, Q_\Reg$ denote the Hausdorff dimensions of $\Sing$ and $\Reg$, respectively)}\label{sumup}
\end{center}
\end{figure}

For generic sub-Riemannian manifolds, we apply our technique  and  we characterize the integrability of the absolutely continuous part by  comparing the dimension of the manifold to dimensions of free Lie algebras (see Proposition~\ref{gensmo}).

As for the singular part $\volh\llcorner_{\Sing}$,  the stratification assumption permits to focus on each equisingular submanifold $N\subset \Sing$. For an equisingular submanifold $N\subset \Sing$, we give an algebraic characterization of the Hausdorff dimension of $N$ and we compare the Hausdorff volume with a smooth volume on the submanifold (see Theorem~\ref{mainth}). We show actually that on the restricted metric space $(N, d|_N)$ the situation is very similar to the one in equiregular manifolds: the Hausdorff volume is absolutely continuous  with respect to any smooth volume on $N$ and we have an expression of the Radon-Nikodym derivative in terms of the nilpotent approximation. Results of this part of the paper have been announced in \cite{noi2}. 

The structure of the paper is the following. In Section~\ref{definizioni} we briefly recall Hausdorff measures on metric spaces and the basic concepts in sub-Riemannian geometry. Section~\ref{lebdec} is devoted to the analysis of the Hausdorff volume in the equiregular case and of  the Lebesgue decomposition of the Hausdorff volume in the non equiregular case.  In Section~\ref{integra} we perform the study of regularity of the Radon-Nikodym derivative and we provide  sufficient
 conditions for integrability and non-integrability. Then, Section~\ref{seqsm} deals with equisingular submanifolds by providing first a complete study of the algebraic and metric structure of such submanifolds, and then an analysis of the properties of the Hausdorff volume.
Finally, in Section~\ref{appli} we apply the methods of Section~\ref{integra} to the study of the generic smooth case and we list some examples. We end with an appendix containing the proof of a technical result, Proposition~\ref{le:unifbbox} which is a uniform Ball-Box Theorem on equisingular submanifolds and which is a key ingredient  to study the behavior of $\frac{d\volh\llcorner_\Reg}{d\mu}$ at singular points.

\section{Definitions}\label{definizioni}

\subsection{Hausdorff and smooth volumes}\label{hhss}

Let us first recall some basic facts on Hausdorff measures. Let $(M,d)$ be
a metric space. We denote by $\diam S$ the diameter of a set $S
\subset M$. 
Let $\alpha \geq 0$ be a real number. For every set $E \subset M$, the
\emph{$\alpha$-dimensional Hausdorff measure} $\hh^\alpha$ of $E$  is
defined as  $\hh^\alpha(E)
= \lim_{\eps \to 0^+} \hh^\alpha_\eps(E)$, where
$$
\hh^\alpha_\eps(E) = \inf \left\{ \sum_{i=1}^\infty  \left(\diam S_i\right)^\alpha
\, : \, E \subset \bigcup_{i=1}^\infty S_i, \ S_i \hbox{ nonempty set} ,
\ \diam S_i \leq \eps \right\},
$$
and the \emph{$\alpha$-dimensional spherical Hausdorff measure} is
defined as $\ss^\alpha(E)
= \lim_{\eps \to 0^+} \ss^\alpha_\eps(E)$, where
$$
\ss^\alpha_\eps(E) = \inf \left\{ \sum_{i=1}^\infty  \left(\diam
S_i\right)^\alpha \, : \, E \subset \bigcup_{i=1}^\infty S_i, \ S_i \hbox{ is
  a ball}, \ \diam
S_i \leq \eps  \right\}.
$$
For every set $E\subset M$, the non-negative number
$$
D=\sup\{\alpha\geq0\mid \hh^\alpha(E)=\infty\}=\inf\{\alpha\geq0\mid \hh^\alpha(E)=0\}
$$
is called the {\it Hausdorff dimension of $E$}. Notice that $\hh^D(E)$
may be $0$, $>0$, or $\infty$.
By construction, for every subset $S \subset M$,
\begin{equation}\label{SHS}
\hh^\alpha(S)\leq\ss^\alpha(S)\leq 2^\alpha\hh^\alpha(S),
\end{equation}
hence the Hausdorff dimension can be defined equivalently using spherical measures.
In the sequel we will call {\it Hausdorff volume} the spherical Hausdorff measure ${\cal S}^{\dim_H M}$ and we will denote this measure by $\volh$.

Given a subset $N\subset M$, we  consider the metric space $(N,d|_N)$. Denoting by $\hh^\alpha_N$ and $\ss^\alpha_N$ the Hausdorff and spherical Hausdorff measures in this space, by definition we have
\begin{eqnarray}
\hh^\alpha\llcorner_N(E)&:=&\hh^\alpha(E\cap N)=\hh_N^\alpha(E\cap N),\nonumber \\
\ss^\alpha\llcorner_N(E)&:=&\ss^\alpha(E\cap N)\leq\ss_N^\alpha(E\cap N)\label{eqes}.
\end{eqnarray}
 Notice that the inequality \r{eqes} is strict in general, as coverings in the definition of $\ss^\alpha_N$ are made with sets $B$ which satisfy $B=B(p,\rho)\cap N$ with $p\in N$, whereas coverings in the definition of $\ss^\alpha\llcorner_N$ include sets of the type  $B(p,\rho)\cap N$ with $p\notin N$  ($\ss^\alpha\llcorner_N$ is very similar to centered Hausdorff measures, see~\cite{edgar07}).   Moreover, thanks to \r{SHS}, for every $S\subset N$, there holds $\hh_N^\alpha(S)\leq\ss_N^\alpha(S)\leq 2^\alpha\hh_N^\alpha(S)$.
Hence
$$
\hh^\alpha(S)\leq\ss_N^\alpha(S)\leq 2^\alpha\hh^\alpha(S),
$$
and $\ss^\alpha_N$ is absolutely continuous with respect to $\hh^\alpha\llcorner_N$.
We will denote by $\volh^N$ the Hausdorff volume $\ss_N^{\dim_HN}$.\medskip

When $M$ is an oriented  manifold we can introduce another kind of volume. We say that a measure $\mu$ is a
 {\it smooth volume} on $M$ if there exists a positively-oriented non degenerate $n$-form $\omega\in\Lambda^nM$ on $M$ (i.e. $\omega$ is a volume form) such that, for every Borel set $E\subset M$, $\mu(E)=\int_E \omega$.

Finally we will say that two measures $\mu, \nu$ on $M$ are {\it commensurable}
if $\mu, \nu$  are mutually absolutely continuous,  i.e., $\mu \ll \nu$ and $\nu\ll\mu$, and if both Radon-Nikodym derivatives $\frac{d \mu}{d \nu}$ and $\frac{d \nu}{d \mu}$ exist and are locally essentially bounded.
When it is the case, for every compact set $K$ there exists $C>0$ such that
$$
\frac1C\mu\llcorner_K \leq \nu\llcorner_K \leq C\mu\llcorner_K.
$$
%
%
In particular, thanks to  \r{SHS},  $\hh^\alpha$ and $\ss^\alpha$ are mutually absolutely continuous and commensurable.

\subsection{Sub-Riemannian manifolds}
\label{se:srman}

%
%
%
%
%

Usually, a sub-Riemannian manifold
  is a triplet  $(M,\Delta,g^R)$, where $M$ is a smooth (i.e., ${\cal C}^\infty$) manifold,    $\Delta$
is a
subbundle of $TM$ of rank $m<\dim
M$ and $g^R$ is a Riemannian metric  on $\Delta$.
Using $g^R$, the length of  horizontal curves, i.e.,
absolutely continuous curves which are almost everywhere tangent to
$\Delta$, is well-defined.
 When $\Delta$ is Lie bracket generating, the map $d:M\times M\to\R$  defined as the infimum of length of horizontal
curves  between two given points is a continuous distance
(Rashevsky-Chow Theorem), and it is called  sub-Riemannian
  distance.

In this paper we study sub-Riemannian manifolds with singularities. Thus it is natural to work in a larger setting, where the map $q\mapsto \Delta_q$ itself may have singularities. This leads us to the following generalized definition \cite{ABB,bellaiche}.

\begin{definition}\label{SR} A {\it sub-Riemannian} structure on a manifold $M$ is a triplet $({\bf U},\ps,f)$ where $({\bf U},\ps)$ is a Euclidean vector bundle over $M$ (i.e., a vector bundle $\pi_{\bf U}:{\bf U}\to M$ equipped with a smoothly-varying scalar product $q\mapsto \langle\cdot,\cdot\rangle_q$ on the fibre ${\bf U}_q$) and $f$ is a morphism of vector bundles $f:{\bf U}\to TM$, i.e. a smooth map linear on fibers and such that, for every $u\in {\bf U}$, $\pi(f(u))=\pi_{\bf U}(u)$, where $\pi:TM\to M$ is the usual projection.
\end{definition}


Let $({\bf U},\ps,f)$ be a sub-Riemannian structure on $M$.  We define the submodule $\bD\subset \VecM$ as
\begin{equation}\label{module}
\bD=\{f\circ\sigma\mid \sigma \textrm{ smooth section of }{\bf U}\},
\end{equation}
 and  for $q \in M$ we set $\bD_q=\{X(q)\mid X\in\bD\}\subset T_qM$.  Clearly $\bD_q=f({\bf U}_q)$.
The  length of a tangent vector $v\in \bD_q$ is defined
as
\begin{equation}\label{qf}
g_q(v):=
\inf\{\langle u,u\rangle_q\mid f(u)=v, u\in {\bf U}_q\}.
\end{equation}
An absolutely continuous  curve $\gamma:[a,b]\to M$ is  {\it horizontal } if $\dot\gamma(t)\in \bD_{\g(t)}$ for almost every $t$.
If $\bD$ is Lie bracket generating, that is
\begin{equation}\label{bg}
\forall\, q\in M\quad \textrm{Lie}_q\bD=T_qM,
\end{equation}
then  the map $d:M\times M\to\R$  defined as the infimum of length of horizontal
curves  between two given points is a continuous distance as in the classic case. In this paper, all sub-Riemannian manifolds are assumed to satisfy the Lie bracket generating condition \r{bg}.

\brem
Definition~\ref{SR} includes the following cases.
\bi
\iii Classic sub-Riemannian structures: in this case ${\bf U}$ is a subbundle of $TM$ and $f$ is the inclusion. With the notations used at the beginning of this section, this amounts to take ${\bf U}=\Delta$, and $\ps=g^R$.  Then the module $\bD$ coincides with the module of smooth sections of the subbundle $\Delta$ and $\bD_q$ has constant dimension.
Moreover, for every $q\in M$ and $v\in\bD_q$, $g_q(v)=g_q^R(v,v)$.
\iii Sub-Riemannian structures associated with a family  of vector fields $X_1,\dots, X_m$: in this case ${\bf U}=M\times\R^m$ is the trivial bundle of rank $m$ over $M$, $\ps$ is the Euclidean scalar product on $\R^m$, and $f:M\times\R^m\to M$ is defined as $f(q,u)=\sum_{i=1}^mu_iX_i(q)$. Here  $\bD$ is the module generated by  $X_1,\dots X_m$ and $g$ is given by
$$
g_q(v)=\inf\left\{\sum_{i=1}^m u_i^2~\Big\vert~  v=\sum_{i=1}^mu_iX_i(q)\right\}.
$$
\ei
\erem

Let $({\bf U}, \ps, f)$ be a sub-Riemannian structure on a manifold $M$, and $\bD$, $g$ the corresponding module and quadratic form as defined in \r{module} and \r{qf}. In analogy with the classic sub-Riemannian case and to simplify notations, in the sequel we will refer to the sub-Riemannian manifold as the triplet $(M, \bD, g)$.  This is justified since all the constructions and definitions below rely only on $\bD$ and
$g$.

%
Given $i\geq 1$, define recursively  the
  submodule $\bD^i\subset \VecM$ by
  $$
\bD^{1}=\bD,\quad \bD^{i+1}=\bD^{i}+[\bD,\bD^{i}].
$$
Fix $p\in M$ and set $\bD^i_p=\{X(p)\mid X\in\bD^i\}$.
The Lie-bracket generating assumption implies that
there exists an integer $r(p)$
such that
 \begin{equation}\label{flagd}
\{0\}= \bD_p^0 \subset\bD_p^1\subset\dots\subset\bD_p^{r(p)}=T_p M.
\end{equation}
The sequence of subspaces \r{flagd} is called the {\it flag of $\bD$ at $p$}.
Set $n_i(p)=\dim\bD^i_p$ and
\begin{equation}
\label{defq}
Q(p)=\sum_{i=1}^{r(p)}i(n_{i}(p)-n_{i-1}(p)).
\end{equation}
 This integer will play a crucial role for determining the Hausdorff dimension of $(M,d)$. To write $Q(p)$ in a different way, we define the {\it weights} of the flag \r{flagd} at $p$ as the integers $w_1(p),\dots, w_{n}(p)$ such that $w_i(p)=s$
if $\dim\bD^{s-1}_p < i \leq \dim\bD^{s}_p$. Then $Q(p)= \sum_{i=1}^n w_i(p)$.

We say that a point $p$ is {\it regular} if, for every $i$, $n_i(q)$ is constant  as $q$ varies in a neighborhood of $p$. Otherwise, the point is said to be {\it singular}.
The sub-Riemannian manifold is called {\it equiregular} if every point is regular.

\medskip

When the dimensions  $n_i(q)$ are constant on  $M$,    the module $\bD^i$ coincides with the  module of vector fields
$$
\{X\in \VecM\mid X(q)\in\bD^i_q~\forall\, q\in M\},
$$
i.e., vector fields that are tangent to the distribution $q\mapsto \bD^i_q$.
Yet the identification
 between the module $\bD^i$ and the map $q\mapsto \bD^i_q$
is no more meaningful when the dimension of $\bD^i_q$ varies as a
function of $q$  (see the discussion in \cite[page 48]{bellaiche}). Indeed, in the rank-varying case, a vector field tangent to $\bD^i_q$ at every $q \in M$ may fail\footnote{For instance, on $M=\R$, take the module $\bD\subset\mathrm{Vec }\, \R$ generated by $X(x)=x^2\partial_x$. Then the  vector field $Y(x)=x\partial_x$ is clearly tangent to the distribution $x\mapsto \bD_x$ but does not belong to $\bD$. } to be in the module $\bD^i$.
Definition~\ref{SR} allows to take account of structures where the dimensions   $n_i(q)$  (and in particular $n_1(q)$)  may vary.
\begin{example}[Grushin plane] Let $M=\R^2$, ${\bf U}=\R^2\times\R^2$ endowed with the canonical Euclidean structure, and $f$ be the morphism defined as follows. If $\sigma_1,\sigma_2$ is a global orthonormal basis on ${\bf U}$, we set $f(\sigma_1(x,y))=\partial_x$ and $f(\sigma_2(x,y))=x\partial_y$. Then $\bD_{(x,y)}$ is two dimensional for every $x\neq 0$, whereas $\dim\bD_{(0,y)}=1$.
\end{example}

\brem
\label{re:Q<}
At a regular point $p$, the equality $n_i(p)=n_{i+1}(p)$ implies that the local distribution  $q\mapsto \bD^i_q$ is involutive, and so that $n_i(p)=n_j(p)$ for any $j \geq i$. From the Lie bracket generating assumption \eqref{bg} we deduce $n_i(p) < n_{i+1}(p)$ for $i < r(p)$, which in turn implies $Q(p)\leq n^2$.
\erem

Given any sub-Riemannian manifold $(M, \bD, g)$ there always exist a (possibly very big) integer $m$ and vector fields $X_1,\dots, X_m$ such that $\bD$ is globally generated by $X_1,\dots, X_m$ and
$$
g_q(v)=\inf\left\{\sum_{i=1}^m u_i^2~\Big\vert~  v=\sum_{i=1}^mu_iX_i(q)\right\} \qquad \hbox{for every } q \in M, \ v \in T_q M.
$$
We call such a family $X_1,\dots, X_m$ a  \emph{(global) generating family} for the sub-Riemannian structure $(\bD, g)$. The existence of a generating family is a consequence of \cite[Corollary~3.16]{ABB}. For an alternative proof see also \cite{suss08}.

Consider a generating family $X_1,\dots, X_{m}$ for
 $(\bD,g)$.
A multi-index $I$ of length $|I|=j\geq 1$ is an element of
 $\{1,\dots,m\}^j$. With any multi-index $I=(i_1,\dots,i_j)$ is
 associated an iterated Lie bracket $X_I=[X_{i_1},[X_{i_2},
 \dots,[X_{i_{j-1}},X_{i_j}]\dots]$ (we set $X_I=X_{i_1}$ if $j=1$). The set of
 vector fields $\{X_I\mid |I|\leq j\}$ generates the module $\bD^j$. As a consequence, if the values of
 $X_{I_1},\dots,X_{I_n}$ at $p\in M$ are linearly independent, then
 $\sum_i |I_i| \geq Q(p)$.

Let $Y$ be a vector field. We define the  {\it length of $\,Y$} by
$$
\ell(Y)=\min\{i\in \N\mid Y\in\bD^i\}.
$$
In particular, $\ell (X_I) \leq |I|$.
By an {\it adapted basis} to the flag \r{flagd} at $p$, we mean $n$ vector fields $Y_1,\dots, Y_n$  such that their values at $p$ satisfy
$$
\bD^i_p=\span\{Y_j(p)\mid  \ell(Y_j)\leq i\},\quad \forall\, i=1,\dots, r(p).
$$
In particular, $\sum_{i=1}^{n}\ell(Y_i) = Q(p)$. As a consequence, a family of Lie brackets $X_{I_1},\dots,X_{I_n}$ such that $X_{I_1}(p),\dots,X_{I_n}(p)$ are linearly independent is an adapted basis to the flag \r{flagd} at $p$ if and only if $\sum_i |I_i| = Q(p)$.

Let $h :U\to\R$ be a continuous function on a neighborhood of $p$. The {\it nonholonomic order of $h$ at $p$} is
$$
\ord_p(h):=\inf\{s\in (0,+\infty) \mid  h(q) = O(d(p,q)^s)\}.
$$
If $h$ is a smooth function, then $\ord_p(h)$ admits an algebraic characterization, namely,
$$
\ord_p(h)=\min\{s\in \N\mid \exists\, i_1,\dots, i_s\in\{1\dots, m\} \textrm{ such that } (X_{i_1}X_{i_2}\cdots X_{i_s}h)(p)\neq 0 \}.
$$
A smooth function $h$ is called  {\it privileged} at $p$ if
$$
\ord_p(h)=\max\{s\in\N\mid dh(\bD_p^{s}/\bD^{s-1}_p)\neq 0\},
$$
where $\bD_p^{s}/\bD^{s-1}_p$ is a vector subspace such that $\bD^s_p= \bD^{s-1}_p\oplus(\bD_p^{s}/\bD^{s-1}_p)$.
We say that coordinates $\varphi=(x_1,\dots,x_n):U\to\R^n$ are {\it privileged at $p$} if they are centered at $p$ and linearly adapted (that is the coordinate vector fields are an adapted basis to the flag at $p$), and if every coordinate function $x_i$ is privileged at $p$.

Note that in this definition of privileged coordinates, the coordinate functions are not ordered by increasing nonholonomic order. However the set  $\{\ord_p(x_i), i=1,\dots,n\}$ coincides with $\{w_1(p),\dots,w_n(p)\}$, so we can relabel the weights in such a way that $\ord_p(x_i)=w_i(p)$. We then say that the weights are \emph{labeled according to the coordinates $\varphi$}.

For every point $p\in M$ the metric tangent cone\footnote{in Gromov's sense, see \cite{gromov}} to $(M,d)$ at $p$ exists and is isometric to $(T_pM,\widehat d_p)$ where $\widehat d_p$ is the sub-Riemannian distance associated with the nilpotent approximation of the structure at $p$ (see \cite[Definition~5.15]{bellaiche} for the definition of nilpotent approximation). This was shown in \cite{mitchell} for equiregular sub-Riemannian manifolds and in \cite{bellaiche} for the general case.
Isometries between the metric tangent cone and  $(T_pM,\widehat d_p)$ are given by   privileged coordinates at $p$.

 By construction, the metric tangent cone is endowed with a family of dilations $\delta_\lambda$ with respect to which the distance $\widehat d_p$ is homogeneous. If $\varphi:U\to\R^n$ is a system of privileged coordinates at $p$ then dilations are given by
$$
\delta_\lambda (x_1,\dots, x_n)=(\lambda^{w_1(p)}x_1,\dots, \lambda^{w_n(p)}x_n),
$$
 where $w_i(p)=\ord_p(x_i)$.

We refer the reader to \cite{ABB,bellaiche,montgomery} for a primer in sub-Riemannian geometry.

\section{Lebesgue decomposition of the Hausdorff volume}\label{lebdec}

In this section we investigate the relation between the Hausdorff volume   and a smooth volume on a sub-Riemannian manifold.
We first recall the equiregular case, where the situation is well understood, then we  consider the case where singular points are present. We write the Lebesgue decomposition of the Hausdorff volume with respect to the smooth volume and  we start the analysis of both the absolutely continuous part and the singular part.

\subsection{Equiregular case}

Assume $(M,\bD,g)$ is an equiregular connected sub-Riemannian manifold. Then $q \mapsto Q(q)$ is constant on $M$ (see \eqref{defq}) and we denote by $Q$ its constant value. Moreover at every point $p\in M$ the metric tangent cone to $(M,d)$, which is  isometric to $(T_pM,\widehat d_p)$, has a structure of a Carnot group. Assume $M$ is also oriented and let $\mu$ be a smooth volume on $M$.
The associated volume form $\omega$   induces canonically a left-invariant volume form $\hat\omega_p$ on $T_pM$. We denote by $\mup$ the smooth volume on $T_pM$ defined by $\hat\omega_p$. To clarify this construction, we refer the reader to Proposition~\ref{densities} in Section~\ref{seqsm}.


\bt\label{mainthreg}
 Let $(M,\bD,g)$ be an equiregular connected oriented sub-Riemannian manifold and let $\mu$ be a smooth volume on $M$. Then
\bi
 \iii[(i)] $\dim_{H}M=Q$ and $\volh=\ss^{Q}$;
\iii[(ii)] $\volh$ is a Radon measure on $M$;
\iii[(iii)] $\volh\ll\mu$ and $\mu\ll\volh$;
\iii[(iv)] the Radon-Nikodym derivative  $\frac{d\volh}{d\mu}(p)$ coincides with  the density $\lim_{\eps\to 0}\frac{\volh( B(p,\eps))}{\mu(
  B(p,\eps))}$, whose value is
    \begin{equation}\label{abscontrr}
\lim_{\eps\to 0}\frac{\volh(B(p,\eps))}{\mu(
  B(p,\eps))}=\frac{2^{Q}}{\mup(\widehat B_p)},~~ \forall\, p\in M,
\end{equation}
where $\widehat B_p$ denotes the ball centered at $0$ and of radius $1$ in $T_pM$ with respect to $\widehat d_p$.
\ei

\brem
 As we will see in Theorem~\ref{mainth} below, one can interpret the constant $2$ in \r{abscontrr} as the diameter of the ball  $\widehat
  B_p$ with respect to the distance $\widehat d_p$. Since the nilpotent approximation is a Carnot group here, we then have $2^Q=\diam_{\widehat d_p}(\widehat B_p)^Q=\ss^{Q}_{\widehat d_p}(\widehat B_p)$ (here $\ss^{Q}_{\widehat d_p}$ denotes the spherical Hausdorff measure in $T_pM$ with respect to the distance $\widehat d_p$), which gives a clear interpretation to \r{abscontrr}.
  \erem
\et

Theorem~\ref{mainthreg} is a special case of  Theorem~\ref{mainth} proved in Section~\ref{seqsm}.
Note that all the statements in Theorem~\ref{mainthreg} are well-known, but they have been proved only fragmentarily in several references: properties (i) and (iii) are stated in \cite{mitchell} (for a rigorous proof see  \cite{montgomery}), property (iv) in \cite{balu}. Up to the authors' knowledge, property (ii) has never been stated as is and it is a consequence of basic covering arguments in geometric measure theory.  In particular, this property is needed to apply the differentiation theorem for Radon measures \cite[Theorem~4.7 p.24]{simon} and ensures that $\frac{d\volh}{d\mu}(p)$ coincides with $\lim_{\eps\to 0}\frac{\volh( B(p,\eps))}{\mu(
  B(p,\eps))}$.
\brem
Properties (i), (ii), (iii)   hold true if we replace $\volh$ by ${\cal H}^{\dim_HM}$.
In particular, applying  the differentiation theorem for Radon measures \cite[Theorem~4.7 p.24]{simon} we get that
(i) the limit $\lim_{\eps\to 0}\frac{\cH^{Q}(B(p,\eps))}{\mu(B(p,\eps))}$ exists $\mu$-almost everywhere and it coincides with the Radon-Nikodym derivative of $\cH^Q$ with respect to $\mu$. Nevertheless, we do not have an explicit representation of such limit as we have for the spherical Hausdorff case in \r{abscontrr}.
\erem
A first consequence of Theorem~\ref{mainthreg} is that the Radon-Nikodym derivative of $\volh$ with respect to $\mu$ is continuous on $M$, see \cite[Corollary~2 and Section~4.1]{balu}. This is due to the fact that both the nilpotent approximation  at $q$  and the tangent measure $\muq$ depend smoothly on $q$. Studying higher regularity of $\frac{d\volh}{d\mu}$ is the main subject in \cite{balu}, to which we refer the interested reader.

A further consequence is that $\mu$ and $\volh$ are commensurable (see the last remark in \cite{mitchell}). This follows directly by the fact that  $q\mapsto\muq(\widehat B_q)$ is positive and continuous on $M$, so that    $\frac{d\volh}{d\mu}(q)$ and its inverse $\frac{d\mu}{d\volh}(q)$ are locally bounded.

In conclusion, when the manifold is equiregular, the Hausdorff volume essentially behaves as a smooth measure.
This fails when singular points are present, as we see in the next section.

\subsection{Non equiregular case}

Assume $(M,\bD,g)$ is an oriented sub-Riemannian manifold and  $\mu$ is a smooth volume on $M$. The manifold is split into the disjoint union of two sets
$$
M=\Reg\cup\Sing,
$$
where  $\Reg,\Sing$ denote respectively the set of regular and singular points. Since the functions $q\mapsto \dim \bD^i_q$ are lower semi-continuous, $\Reg$ is an open and dense subset of $M$ and conversely $\Sing$ is a closed subset of empty interior.

We will assume in this paper that $\mu(\Sing)=0$. This assumption is satisfied for a very wide class of sub-Riemannian manifolds: for instance for analytic sub-Riemannian manifolds, for generic sub-Riemannian structures on a given manifold (see Section~\ref{gsc}), or when $q \mapsto \bD_q$ is a distribution of corank 1 \cite[Proposition~3.3]{Rifford2014}.
 Nevertheless one can build up sub-Riemannian manifolds where $\mu(\Sing) \neq0$ as in the next example.

\begin{example} 
Let $S\subset \R^4$ be any closed set having empty interior and positive Lebesgue measure.
By Whitney's Extension Theorem (see for instance \cite[Proposition~A.8]{Madsen1997}), there exists a function $f\in\con^\infty(\R^4)$   such that $S=f^{-1}(0)$.
Consider the sub-Riemannian structure on $\R^4$ for which the following vector fields are a generating family,
$$
X_1=\partial_1,~~X_2=\partial_2+x_1\partial_3+\frac{x_1^2}{2}\partial_4,~~X_3=f(x)\partial_4.
$$
One easily checks that
the singular set coincides with $S$ and thus it is of positive measure.
\end{example}

 We are interested in properties of  $\volh=\ss^{\dim_HM}$, hence we first compute $\dim_HM$. Clearly, the Hausdorff dimension of $M$ is $\max\{\dim_H\Reg,\dim_H\Sing\}$.
Then, using  property (i) in Theorem~\ref{mainthreg} the Hausdorff dimension of the regular set can be computed algebraically as
$$
\dim_H\Reg= \sup\{\dim_{H}O, \  O \textrm{ connected component of } \Reg\}=\max_{q\in\Reg}Q(q).
$$
This follows from the fact that $\Reg$ is a countable  union of open connected components, and by the inequality $Q(q)\leq n^2$ at every $q \in \Reg$ (see Remark~\ref{re:Q<}).
To remind its algebraic characterization, we denote   by $Q_\reg$ the number $\max_{q\in\Reg}Q(q)$ (which equals $\dim_H\Reg$). We present in Section~\ref{seqsm} an analogous method to compute the Hausdorff dimension of $\Sing$ under some stratification assumptions.

Applying Theorem~\ref{mainthreg} to every connected component of $\Reg$ we obtain the following result.

\begin{proposition}
\label{rnvolh}
Let $(M,\bD,g)$ be an oriented sub-Riemannian manifold and let $\mu$ be a smooth volume on $M$. Assume $\mu(\Sing) = 0$. Then $\ss^{Q_\reg}\llcorner_\Reg\ll\mu$, the Radon-Nikodym derivative of $\ss^{Q_\reg}\llcorner_\Reg$ with respect to $\mu$ exists $\mu$-almost everywhere, and it coincides with the density $\lim_{\eps\to 0}\frac{\ss^{Q_\reg}(B(q,\eps))}{\mu(B(q,\eps))}$ for every $q \in \Reg$.

Moreover, if $O$ is a connected component of $\Reg$ such that $\dim_HO =Q_\reg$, then
\begin{equation}
\label{abscontreg}
\frac{d \ss^{Q_\reg}\llcorner_\Reg}{d \mu}(q)=  \frac{2^{Q_\reg}}{\muq( \widehat B_q)} \qquad \hbox{for every } q \in O.
\end{equation}
\end{proposition}

\brem
The measure $\ss^{Q_\reg}\llcorner_\Reg$ may not be a Radon measure on $M$ (see Section~\ref{integra} below). Thus the existence of a Radon-Nikodym derivative and the fact that it coincides with the density are not consequences of $\ss^{Q_\reg}\llcorner_\Reg\ll\mu$.
\erem

\brem
In the statement above we can replace $\ss^{Q_\reg}$ by $\volh\llcorner_\Reg$. Indeed, if  $\dim_H\Sing >Q_{\reg}$, then $\volh=\ss^{\dim_H\Sing}$. As a consequence, $\volh\llcorner_{\Reg}\equiv 0$ and $\volh\llcorner_{\Reg}(B(q,\eps))=0$, for every $q\in \Reg$, so that the density becomes trivial. On the other hand, if $\dim_H\Sing \leq Q_{\reg}$ then $\volh\llcorner_{\Reg}=\ss^{Q_\reg}\llcorner_\Reg$, that is, the two measures coincides.
\erem

\brem
When there is a connected component $O\subset\Reg$ with $\dim_HO <Q_\reg$, the density $\lim_{\eps\downarrow 0}\frac{\ss^{Q_\reg}(B(q,\eps))}{\mu(B(q,\eps))}$  vanishes at every point $q\in  O$ and so $\frac{d \ss^{Q_\reg}\llcorner_\Reg}{d \mu} \equiv $ on $O$.
\erem

\begin{proof}[Proof of Proposition~\ref{rnvolh}]
 Apply Theorem~\ref{mainthreg} to  each connected component  of $ \Reg$. Then, by property (ii),  $\ss^{Q_\reg}$, as a measure on the metric space $(\Reg,d|_{\Reg})$, is Radon, that is, $\ss^{Q_\reg}$ is finite on compact subsets  contained in  $\Reg$. Moreover, by property (iii), $\ss^{Q_\reg}$, as a measure on $\Reg$, is absolutely continuous with respect to $\mu$. Obviously, $\mu$ is also a Radon measure on $M$ (and thus on $\Reg$) which is also locally doubling (see \cite{nsw}).
Thus we can apply the differentiation theorem for Radon measures (see for instance \cite[Theorem~4.7 p.24]{simon} with $X=\Reg$, $\mu_2=\volh$ and $\mu_1=\mu$) and deduce that, on $\Reg$,   the Radon-Nikodym derivative of $\ss^{Q_\reg}$ with respect to $\mu$ coincides with the density $\lim_{\eps\to 0}\frac{\ss^{Q_\reg}(B(q,\eps))}{\mu(B(q,\eps))}$.
In other words, for every Borel set $E\subset \Reg$,
$$
\ss^{Q_\reg}(E)=\int_E\lim_{\eps\to 0}\frac{\ss^{Q_\reg}(B(q,\eps))}{\mu(B(q,\eps))}d\mu(q).
$$
Since $\Reg$ is open and $\mu(M\setminus \Reg) =0$,  we deduce that for every Borel set $E\subset M$,
$$
\ss^{Q_\reg}\llcorner_\Reg(E)=\int_{E\cap \Reg}\lim_{\eps\to 0}\frac{\ss^{Q_\reg}(B(q,\eps))}{\mu(B(q,\eps))}d\mu(q),
$$
 which gives the conclusion.
\end{proof}


Since $\Reg$ is open and Hausdorff measures are Borel regular, $\Reg$ and $\Sing$ are  $\mu$- and $\volh$-measurable.
Hence for every set $E\subset M$
$$
\volh(E)=\volh(E\cap \Reg)+\volh(E\cap\Sing),
$$
or equivalently, $\volh=\volh\llcorner_{\Reg}+\volh\llcorner_{\Sing}$.
Moreover, since $\mu(\Sing)=0$,  $\volh\llcorner_\Sing$ is concentrated on $\Sing$ whereas $\mu$ is concentrated on $\Reg$. Therefore $\mu$ and $\volh\llcorner_\Sing$ are mutually singular. We thus get directly the following fact.

\begin{corollary}\label{jkl}
Let $(M,\bD,g)$ be an oriented sub-Riemannian manifold and let $\mu$ be a smooth volume on $M$. Assume $\mu(\Sing)=0$. Then $\mu$ and $\volh\llcorner_{\Sing}$ are mutually singular and  the Lebesgue  decomposition of $\volh$ with respect to $\mu$ is
\begin{equation*}\label{dec}
\volh=\volh\llcorner_{\Reg}+\volh\llcorner_{\Sing}.
\end{equation*}
As a consequence,
(i) if  $\dim_H\Reg<\dim_H \Sing$ then  $\volh$ and $\mu$ are mutually singular;
(ii) if $\dim_H\Reg>\dim_H\Sing$ then $\volh\ll\mu$.
\end{corollary}

Corollary~\ref{jkl} provides the Lebesgue decomposition of $\volh$ with respect to a smooth measure $\mu$. In the sequel we are interested in studying the absolutely continuous part and the singular part of $\volh$.
 Note that the only case where $\volh\llcorner_\Reg\not\equiv 0$ is when $\dim_H\Reg\geq \dim_H\Sing$. In this case $\volh\llcorner_\Reg=\ss^{Q_\Reg}\llcorner_\Reg$. Thus the latter is the measure we study in Section~\ref{integra}.

\subsection{Stratification assumption}
\label{se:stratif}

To go further in the characterization of $\volh$, the first question is how to compute the Hausdorff dimension $\dim_H M$ and how to relate $\dim_H M$
to algebraic properties of the sub-Riemannian structure. 

Recall that to compute $\dim_H\Reg$ one simply considers $\Reg$ as the disjoint union of open sets where the growth vector $(n_1(q),\dots,n_{r(q)}(q))$ is constant
and then compute the Hausdorff dimension of each component $O\subset\Reg$ by the algebraic formula
$$
\dim_H O=\sum_{i=1}^{r(q)}i(n_i(q)-n_{i-1}(q)).
$$
This idea can be carried out on the whole manifold, provided that $M$ can be stratified by suitable submanifolds. This motivates the following definition.

Let $N\subset M$ be a smooth connected submanifold and $q \in N$. The \emph{flag  at $q$ of $\,\bD$ restricted to $N$} is the sequence of subspaces
\begin{equation} \label{flagq}
\{0\} \subset (\bD_q^1\cap T_qN)\subset\dots\subset(\bD_q^{r(q)}\cap T_qN)=T_qN.
\end{equation}
Set
\begin{eqnarray}
& n_i^N(q)= \dim (\bD_q^i\cap T_qN)\qquad \hbox{and} \qquad
Q_N(q)=\sum_{i=1}^{r(q)}i(n^N_{i}(q)-n^N_{i-1}(q)).& \label{defqn}
\end{eqnarray}
 We say that
$N$ is {\it equisingular}\footnote{Note that in our previous paper \cite[Definition~1]{noi2}, these submanifolds are called {\it strongly equiregular}.} if
\bi
\iii[(i)] for every $i$,  the dimension
 $n_i(q)$ is constant as $q$ varies in $N$;
\iii[(ii)] for every $i$,  the dimension
 $n_i^N(q)$ is constant as $q$ varies in $N$.
 \ei
In this case, we denote by $Q_N$  the constant value of $Q_N(q)$,
 $q\in N$. We will see in Theorem~\ref{mainth} that $Q_N$ is actually the Hausdorff dimension of $N$.

\brem
If $N$ is an open connected submanifold of $M$ then $N$ is equisingular if and only if it is equiregular, that is, condition (i) is equivalent to  condition (ii). For submanifolds $N$ of dimension  smaller than $\dim M$ being equisingular implies that $N$ is contained in the singular set (see Example~\ref{mart} below for a sub-Riemannian structure where the singular set is itself an equisingular submanifold).
\erem

\begin{definition}
\label{sequi}
We say that the sub-Riemannian manifold $(M,\bD,g)$ is \emph{stratified by equisingular submanifolds} if there exists a locally finite stratification $\Sing=\cup_i\Sing_i$ where every $\Sing_i$ is an equisingular submanifold.
\end{definition}

In the rest of the paper we will make the assumption that $(M,\bD,g)$ is stratified by equisingular submanifolds. The interest of such an assumption has already been raised by Gromov \cite[1.3.A]{gro96}.
It holds true in particular for generic smooth sub-Riemannian manifolds and for analytic ones. It obviously implies $\mu(\Sing)=0$ for any
smooth measure $\mu$, hence Corollary~\ref{jkl} and all the analysis of the previous section are valid. To shorten the notations we gather all the hypotheses we need on the sub-Riemannian manifold in the following assumption:
\begin{description}
  \item[(A)] The sub-Riemannian manifold $(M,\bD,g)$ is stratified by equisingular submanifolds, $M$ is oriented, and $Q(q)=Q_\Reg$ at every point $q \in \Reg$.
\end{description}
\medskip

As a first consequence of this  assumption, we obtain $\dim_H\Sing =\sup_i \dim_H{\Sing_i}$
 and $\volh\llcorner_\Sing = \sum_i \volh\llcorner_{\Sing_i}$.   Thus, to characterize $\volh\llcorner_S$, it suffices to understand what happens on each stratum. This will be done Section~\ref{seqsm} where
 we develop the analysis of Hausdorff volumes on equisingular submanifolds started in \cite{noi2}.\medskip

As concerns the regular part of the Hausdorff volume, note that all connected components $O$ of $\Reg$ have the same Hausdorff dimension $Q_\reg$,  which implies that
\eqref{abscontreg} holds for every $q \in \Reg$. Also, it makes sense to study $\ss^{Q_\reg}\llcorner_\Reg$ rather than
$\volh\llcorner_\Reg$. Indeed, if $\dim_H\Sing\leq Q_\reg$ then  $\volh=\ss^{Q_\reg}$.
 Conversely,  if $\dim_H\Sing >Q_\reg$ then the absolutely continuous part $\volh\llcorner_\Reg$ vanishes identically, but $\ss^{Q_\reg}\llcorner_\Reg$ still satisfies Proposition~\ref{rnvolh}.

We will focus our analysis on the regularity properties of the derivative  $\frac{d\ss^{Q_\reg}\llcorner_\Reg}{d\mu}$ as they translate directly into properties  of the measure $\ss^{Q_\reg}\llcorner_\Reg$. In particular, $\ss^{Q_\reg}\llcorner_\Reg$   is commensurable to  $\mu$ if and only if $\dSdm$ is locally essentially bounded and bounded away for zero    and  $\ss^{Q_\reg}\llcorner_\Reg$ is a Radon measure if and only if $\dSdm$ is locally integrable with respect to $\mu$.
To carry out the analysis of this Radon-Nikodym derivative the main tool is the proposition below.

For simplicity,  we introduce the following notations. For two real numbers $t$ and $t'$ we write $t \asymp_C t'$ if $t'/C \leq t \leq C t'$, where $C>0$ is a constant. Moreover, given a family of vector fields $X_1,\dots, X_m$ and a $n$-tuple  $\cI=(I_ 1,\dots, I_n)$ of multi-indices we denote by $\cX_{\cI}$ the $n$-tuple of vector fields $(X_{I_1},\dots, X_{I_n})$ (where $X_{I_j}$ is the Lie bracket corresponding to the multi-index $I_j$).

\begin{proposition}\label{we}
Assume (A) and consider  a generating family $X_1,\dots, X_m$  for the sub-Riemannian structure.
For every compact subset $K \subset M$, there exists a constant
$C>0$ such that
$$
\dSdm(q) \asymp_C \frac{1}{\nu(q)}\quad \forall\,q \in  K\cap \Reg,
$$
where
\begin{equation*}\label{nuu}
{\nu(q)} = \sqrt{\sum_{\cI\in\cF}\omega(\cX_\cI)(q)^2 } \qquad \hbox{and} \qquad \cF=\left\{\cI=(I_1,\dots, I_n)~\left|~\sum_i|I_i|=Q_\reg\right.\right\}.
\end{equation*}
\end{proposition}
The proof is postponed to  Section~\ref{fghjk} . Actually, since $\dSdm(q)=\frac{2^{Q_\reg}}{\muq(\widehat B_q)}$ for $q\in\Reg$, Proposition~\ref{we} follows from  a particular instance of Proposition~\ref{th:nuq}, see also Remark~\ref{ggg}.

\brem
\label{re:ordnu}
Since the family $\cF$ is finite, one easily computes the nonholonomic order of the function $q\mapsto\nu(q)$ at $p\in \Sing$ as
\begin{equation*}
\ord_p\nu=\min_{\cI\in\cF}\ord_p(\omega(\cX_\cI)).
\end{equation*}
 Moreover,   in the statement of Proposition~\ref{we}, $\nu$  can be replaced by
$$
\bar \nu(q)=\max_{\cI\in\cF}|\omega(\cX_{\cI})(q)|.
$$
\erem

\begin{example}[the Martinet space]\label{mart}
Let us compute $\nu$ on a specific example. Consider the sub-Riemannian structure on $\R^3$ given by the generating family
$$
X_1=\partial_1,\quad X_2=\partial_2+\frac{x_1^2}{2}\partial_3.
$$
We choose $\omega=dx_1\wedge dx_2\wedge dx_3$, that is, the canonical volume form on $\R^3$.
 The growth vector is equal to $(2,2,3)$ on the plane $\{ x_1=0\}$, and it is  $(2,3)$ elsewhere. As a consequence, $\Sing$ coincides with  $\{ x_1=0\}$  and is equisingular, since   $n_1^\Sing(0,x_2,x_3)=n_2^\Sing(0,x_2,x_3)=1, n_3^\Sing(0,x_2,x_3)=2$.  At regular points we have $Q_\reg= 4$ and the only adapted basis at regular points is   $(X_1, X_2, [X_1,X_2])$. Then
$$
\nu(x_1,x_2,x_3)\asymp_C\bar\nu(x_1,x_2,x_3)=|x_1|.
$$
\end{example}

By Proposition~\ref{we}, on $\Reg$ the Radon-Nikodym derivative $\dSdm$ is  locally    equivalent to $1/\nu$.
Henceforth it is sufficient to study  essential boundedness and integrability of $1/\nu$.
It turns out that this function is not locally essentially bounded around the singular set.

\begin{proposition}\label{bdd}
Under assumption (A), for every $p\in \Sing$ and every neighborhood $U\subset M$  of $p$, 
$$
\displaystyle{
\mathrm{essup}\dSdm \big|_{U}=+\infty.
}
$$
As a consequence, $\ss^{Q_\reg}\llcorner_\Reg$ and $\mu$ are not commensurable when $\Sing\neq\emptyset$.
\end{proposition}

\begin{proof}
 Let $\cI=(I_1,\dots, I_n)\in\cF$. Recall   that if the values of
 $X_{I_1},\dots,X_{I_n}$ at a point $q\in M$ are linearly independent, then
 $\sum_i |I_i|=Q_\reg \geq Q(q)$ (see Section~\ref{se:srman}). Let $p\in \Sing$. Since $Q_\reg <Q(p)$,
 the vectors $X_{I_1}(p), \dots, X_{I_n}(p)$ are linearly dependent.  Thus $\nu(q)\to 0$ as $q \to p$ and the conclusion follows from Proposition~\ref{we}.
\end{proof}

Thus $\ss^{Q_\reg}\llcorner_\Reg$ and $\mu$ are never commensurable in non equiregular sub-Riemannian manifolds. It makes sense now to ask  whether $\dSdm$ is locally integrable with respect to $\mu$. The answer to this question is more involved and is the object of
Section~\ref{integra}.

\brem
Let us mention that, on $\Reg$, there is a smooth volume which is intrinsically associated with the sub-Riemannian structure (i.e., to the pair $(\bD,g)$), namely the Popp measure $\cP$. An explicit formula for this measure is given in \cite{br}. More precisely, thanks to \cite[Theorem~1]{br}, the Radon-Nikodym derivative $\frac{d\cP}{d\mu}$ satisfies
$$
\frac{d\cP}{d\mu}(q) \asymp_C \frac{1}{\nu(q)}, \qquad \hbox{for every } q\in K\cap\Reg,
$$
where $K$ is any compact set in $M$.  As a consequence, $\cP$ and $\ss^{Q_\reg}\llcorner_\Reg$ are  commensurable on $M$ ($\cP$ is extended to a measure on $M$ by setting $\cP \llcorner_\Sing =0$). In particular,  $\cP$ and $\mu$ are not commensurable when $\Sing\neq\emptyset$, and $\cP$ is Radon if and only if $\ss^{Q_\reg}\llcorner_\Reg$ is.
\erem

\brem\label{ess} In Proposition~\ref{we} and Proposition~\ref{bdd}, the only conditions of assumption (A) needed are $Q(q)=Q_\reg$ for every $q\in\Reg$ and $M$ oriented. The stratification assumption is redundant.
\erem


\section{Local integrability  of the Radon-Nikodym derivative}\label{integra}

Thanks to Theorem~\ref{mainthreg}, $\dSdm$ is continuous on $\Reg$ and thus  locally $\mu$-integrable on $\Reg$. Nevertheless, by Proposition~\ref{bdd} $\dSdm$ explodes when approaching $\Sing$. Therefore, the question remains whether  $\dSdm$ is locally $\mu$-integrable  around singular points, or equivalently whether $\ss^{Q_\Reg}\llcorner_\Reg$ is a Radon measure.
 The aim of this section is to provide sufficient conditions for  $\dSdm$ to be locally $\mu$-integrable or not around $\Sing$.

In Section~\ref{asdf} we express in some suitable coordinates an integral whose finiteness is equivalent to the local integrability of  $\dSdm$. We then use this expression in Section~\ref{tyui} to deduce explicit algebraic conditions for both integrability and  non-integrability of $\dSdm$.


\subsection{Computation in local coordinates}
\label{asdf}

Consider an equisingular submanifold $N$ and a point $p\in N$.
Thanks to Proposition~\ref{we}, under assumption (A) the  local $\mu$-integrability of $\dSdm$ near $p$ is equivalent to the one of $\frac{1}{\nu}$. The latter can be conveniently characterized in well-chosen systems of coordinates.


\begin{lemma}
\label{tech} Assume (A).
Let $N$ be an equisingular $k$-dimensional submanifold, $p\in N$,  and  $\varphi=(y,z) \in\R^{k}\times\R^{n-k}$ be  local coordinates centered at $p$ such that, near $p$,
  $$
  \varphi(N) \subset \{ z=0\}.
  $$
Then $\dSdm$ is locally $\mu$-integrable near $p$ if and only if, for  $R>0$ small enough,
\begin{equation}
\label{eq:eucl}
\int_0^{R}\lambda^{n-k-1} \, \left(\int_{|y|\leq {R}}\left(\int_{\{|z|=1\}}\frac{d\sigma(z)}{\nu(\Phi(y,\lambda z))}\right) dy \right)  d\lambda < + \infty,
\end{equation}
where $\Phi=\varphi^{-1}$, $|y|=\max_{i\in\{1,\dots, k\}}|y_i|$, $|z|= \max_{i\in\{k+1,\dots, n\}}|z_i|$, 
$\lambda z=(\lambda z_{k+1},\dots, \lambda z_{n})$, and $d\sigma(z)$ denotes the $(n-k-1)$-Lebesgue measure on $\{|z|=1\}$.


If moreover $\varphi$ is a system of privileged coordinates at $p$, then $\dSdm$ is locally $\mu$-integrable near $p$ if and only if, for  $R>0$ small enough,
\begin{eqnarray}
\label{plm}
 \int_0^{R}\lambda^{Q(p)-Q_N-1}\,\left(  \int_{\|y\|\leq {R}} \left( \int_{\{\|z\|=1\}} \frac{d\tilde{\sigma}(z)}{\nu(\Phi(y,\delta_\lambda z))} \right)dy  \right) \,d\lambda & < & + \infty ,
\end{eqnarray}
where
$$
 \|y\|=\max_{i\in\{1,\dots, k\}}|y_i|^{w_i(p)}, \quad\|z\|= \max_{i\in\{k+1,\dots, n\}}|z_i|^{w_i(p)}, \qquad
\delta_\lambda z= (\lambda^{w_{k+1}(p)}z_{k+1},\dots,\lambda^{w_{n}(p)}z_{n}),
$$
$w_i(p)$ are  the weights at $p$ labeled according to the coordinates $\varphi$, and $d \tilde{\sigma}(z)$ denotes the $(n-k-1)$-Lebesgue measure on $\{\|z\|=1\}$.
\end{lemma}

\brem
\label{re:QN}
Since $N$ is equisingular, the functions $q \mapsto w_i(q)$ are constants on $N$ and the integers $Q(q)$  and $Q_N=Q_N(q)$ (see \eqref{defqn}) may be computed as
$$
Q(q) = \sum_{i=1}^n w_i(q) \qquad \hbox{and} \qquad Q_N = \sum_{i=1}^k w_i(q).
$$
\erem

\begin{proof}
According to the notation in the statement above,  we set
$\varphi=(y_1,\dots,y_k, z_{k+1},\dots,z_n)$. This allows to  distinguish which coordinates parameterize $N$, namely the first group $y$, and which ones are transversal to $N$, namely the second group $z$.

Let $U$ be a small neighborhood of $p$. Thanks to the expression of $\dSdm$ provided by Proposition~\ref{we}, the finiteness of $\ss^{Q_\reg}\llcorner_{\Reg}(U)$ reduces to the finiteness of the integral $\int_{U}\frac{1}{\nu(q)}d\mu(q)$.
Applying the change of variables $\varphi$ in the integral and taking upper and lower bounds for $|\det J\Phi (y,z)|$ near $(0,0)$, we obtain
\begin{eqnarray*}
\ss^{Q_\reg}\llcorner_{\Reg}(U) &\asymp_{C}&\int_{\varphi(U)}\frac{1}{\nu(\Phi(y,z))}\,dy dz = \int_{y\in\Omega_1}\left(\int_{z\in \Omega_2}\frac{1}{\nu(\Phi(y,z))}dz\right)dy,
\end{eqnarray*}
where we have written $\varphi(U)$ as $\Omega_1 \times \Omega_2$, with $\Omega_1 \subset \R^k$ and $\Omega_2 \subset\R^{n-k}$ open subsets.

Since $\varphi(U)$ contains and  is contained in sets of the form $\{ |y| <R \}\times \{ |z| <R \}$, it is sufficient to consider the case where $\Omega_1 = \{ |y| <R \}$ and $\Omega_2 = \{ |z| <R \}$
for some small $R>0$. Criterion~\eqref{eq:eucl} follows directly by a change of coordinates $z = \lambda \tilde{z}$, $|\tilde{z}| = 1$,  in the integral on $\Omega_2$.

Assume now that the coordinates $\varphi$ are privileged at $p$. As above, it is sufficient to consider the case where $\Omega_1 = \{ \|y\| <R \}$ and $\Omega_2 = \{ \|z\| <R \}$ for some small $R>0$.
Recall that in privileged coordinates, dilations take the form
$$
\delta_\lambda (y,z)=(\lambda^{w_1(p)}y_1,\dots,\lambda^{w_{k}(p)}y_{k}, \lambda^{w_{k+1}(p)}z_{k+1},\dots,\lambda^{w_{n}(p)}z_{n}).
$$
By construction, the pseudo-norm $\|\cdot\|$ is homogeneous with respect to $\delta_\lambda$, and we have
$
\Omega_2=\{\delta_\lambda z \mid \|z\|=1, \lambda\in[0,R)\}.
$
Thus, changing variables in the integral, we obtain
\begin{equation*}\label{hhhh}
 \int_{y\in\Omega_1}\left(\int_{z\in \Omega_2}\frac{1}{\nu(\Phi(y,z))}dz\right)dy=\int_{\|y\| \leq R}\left(\int_0^R\lambda^{\sum_{j=k+1}^n w_{j}-1}\left(\int_{\{\|z\|=1\}}\frac{d\sigma(z)}{\nu(\Phi(y,\delta_\lambda z))}\right) d\lambda\right)dy. 
\end{equation*}
By Remark~\ref{re:QN}, $\sum_{j=k+1}^nw_{j}=Q(p)-Q_N$ and we obtain criterion \eqref{plm}.
\end{proof}
As pointed out in Remark~\ref{ess}, in Lemma~\ref{tech} we do not need the stratification assumption, but only the other two conditions in assumption (A). Nevertheless, whenever $M$ is stratified by equisingular submanifolds and $p\in M$, there exists a unique equisingular submanifold $N$ containing $p$.


We end this section  with the construction of a system of coordinate  around $p$ satisfying the assumptions of the previous lemma. This construction will also prove to be useful in Section~\ref{seqsm}.

\bl
\label{cp} Let $N$ be an equisingular submanifold.
Then there exist a neighborhood $U$ of $p$ and a smooth map
\begin{eqnarray*}
\varphi: U\cap N&\to& {\cal C}^{\infty}(U,\R^n)\\
q&\mapsto& \varphi_q:U\to\R^n
\end{eqnarray*}
 such that, for every $q\in U\cap N$,
 \bi
\iii[$(\mathrm{I})$] $\varphi_q$ is a system of privileged coordinates centered at $q$;
\iii[$(\mathrm{II})$] $\varphi_q(U\cap N) \subset \{(x_1,\dots,x_n)\in\R^n \mid x_{k+1}=\dots=x_n=0\}$, where $k=\dim N$.
\ei
Moreover, for every $q\in U\cap N$ and for every $q'\in U$,
\begin{equation}\label{pro}
(\varphi_q(q'))_{i}=(\varphi_{p}(q'))_{i}, \quad i=k+1, \dots, n.
\end{equation}
\el

\begin{proof}
For every $q\in N$, let us consider the flag of $\bD$ restricted to $N$
$$
(\bD_q^1\cap T_qN)\subset\dots\subset(\bD_q^{r}\cap T_qN)=T_qN.
$$
Recall that  the integers $n_i^N=\dim(\bD^i_q\cap T_qN)$ do not depend on $q$ in $N$. Thus, in a  neighborhood $U$ of $p$, there exist a family of vector fields $Z_1,\dots, Z_k$ on $U$  such that, for every $q\in N$ and for every $i=1,\dots, r$,
  $$
\bD^i_q\cap T_qN=\span\{Z_j(q)\mid   \ell(Z_j)\leq i \}.
$$
In particular, we have $T_qN=\span\{Z_1(q),\dots, Z_k(q)\}$,   and
$$
\qn=\sum_{i=1}^{r}i(n^N_{i}-n^N_{i-1})=\sum_{i=1}^{k}\ell(Z_i).
$$
Possibly reducing $U$, since $q\mapsto \dim\bD^i_q$ is constant on $N$, there exist
vector fields $Z_{k+1},\dots, Z_{n}$ such that the family $Z_1,\dots, Z_n$ is adapted to the flag of $\bD$  at every $q \in N \cap U$ (see  \r{flagd}).

Using these bases, we define
for   $q \in  N \cap U$, a local diffeomorphism
$\Phi_q:\R^n\to M$ by
\begin{equation}\label{coorn}
\Phi_q(x)=\exp\left(x_nZ_n\right)\circ\dots\circ\exp\left(x_{1}Z_{1}\right)(q).
\end{equation}
 The inverse $\varphi_q=\Phi_q^{-1}$ of $\Phi_q$ provides a system of
 privileged coordinates at
 $q$ (see \cite{hermes}) and in these coordinates $N$ coincides locally with the set $\left\{
   x_{k+1}=\dots=x_n=0 \right\}$.  Moreover, the map $(q,x) \mapsto\Phi_q(x)$ is smooth on $N \cap U$, which completes the proof of points (I) and (II).

Finally, let us prove \r{pro}.  Let $q\in U\cap N$ and $q'\in U$. We denote respectively by $x$ and $y$ the coordinates of $q'$ centered at $q$ and $p$, that is,
$q'=\Phi_q(x)=\Phi_p(y)$. By construction, the point $\exp\left(x_{k}Z_{k}\right)\circ\dots\circ\exp\left(x_{1}Z_{1}\right) (q)$
belongs to $U\cap N$, and hence may be written as $\Phi_p(z)$ with $z_{k+1} = \cdots = z_n=0$. Therefore,
$$
q' =\Phi_q(x) = \exp\left(x_{n}Z_{n}\right)\circ\dots\circ\exp\left(x_{k+1}Z_{k+1}\right)\circ\exp\left(z_{k}Z_{k}\right)\circ\dots\circ\exp\left(z_{1}Z_{1}\right)(p),
$$
which implies $y_i = x_i$ for $i=k+1, \dots, n$, i.e. \r{pro}.
\end{proof}


\subsection{Sufficient conditions}\label{tyui}
In the sequel, we provide sufficient conditions to obtain either local integrability or non integrability of $\frac{d\ss^{Q_\Reg}\llcorner_{\Reg}}{d\mu}$ around a fixed singular point. These conditions will be of different nature (either differential or metric), depending on whether they are obtained with \eqref{eq:eucl} or \eqref{plm}.

Note that by assumption (A), $\Sing$ admits a stratification by equisingular submanifolds, hence  every singular point belongs to an equisingular submanifold which is  determined in a unique way. Recall also from Proposition~\ref{we} that
\begin{equation*}
{\nu(q)} = \sqrt{\sum_{\cI\in\cF}\omega(\cX_\cI)(q)^2 } \qquad \hbox{and} \qquad \cF=\left\{\cI=(I_1,\dots, I_n)~\left|~\sum_i|I_i|=Q_\reg\right.\right\}.
\end{equation*}

\subsubsection{Conditions on nonholonomic orders}
Consider a singular point $p\in\Sing$ and  an equisingular submanifold $N \subset \Sing$ containing $p$. Define
$$
\varrho_{\min}(p,N)=\lim_{\eps\to 0}\min\{\mathrm{ord}_{q}\nu\mid {q\in N\cap B(p,\eps)}\},
$$
where $\ord_q(h)$ is the nonholonomic order of the function $h$ at $q$.
Note  that by Remark~\ref{re:ordnu},    $\ord_p\nu=\min_{\cI\in\cF}\ord_p(\omega(\cX_\cI))$, so that
$$
\varrho_{\min}(p,N)=\lim_{\eps\to 0}\min\{\mathrm{ord}_{q}(\omega(\cX_\cI))\mid {q\in N\cap B(p,\eps)}, \ \cI\in\cF \}.
$$

Next result gives an algebraic sufficient condition for a neighborhood of $p$ to have infinite $\ss^{Q_\Reg}$ measure. This condition involves the order of  $\nu$ at singular points around $p$ ($\varrho_{\min}(p,N)$), the flag of the distribution at $p$ ($Q(p)$) and the flag of the distribution restricted to the singular set at $p$ ($Q_N$).

\bp
\label{fin}
Assume (A). Let $p\in\Sing$ and let $N$ be an equisingular submanifold containing $p$. If $
 \varrho_{\min}(p,N) \geq Q(p)-Q_N$,  then $\frac{d\ss^{Q_\Reg}\llcorner_{\Reg}}{d\mu}$ is not $\mu$-locally integrable around $p$.
\ep

\begin{proof}
We will apply condition~\eqref{plm} with the coordinate system  $\varphi_p=(y,z)$ built in Lemma~\ref{cp}.
Fix $R$ small enough.
 By definition,
$$
\nu(\Phi_p(y,\delta_\lambda z))^2=\sum_{\cI\in{\cal F}} \omega(\cX_{\cI})({\Phi_p(y,\delta_\lambda z)})^2,
$$
where $\cF$ is the family of $n$-tuples of multi-indices $\cI=(I_1,\dots, I_n)$ such that $\sum_j|I_j|=Q_\Reg$.
Since $p$ belongs to the singular set, for any family $\cI\in{\cal F}$, the vectors $X_{I_1}(p),\dots, X_{I_n}(p)$ are linearly dependent. Similarly, for every $y$ near $0$, the vectors $X_{I_1}(\Phi_p(y,0)),\dots, X_{I_n}(\Phi_p(y,0))$ are linearly dependent. As a consequence, $\nu({\Phi_p(y,0)})\equiv0$.

 Given   $y \in \{ \|y\| \leq R \}$, denote by $\varrho(y)$ the greatest number such that there exists a constant $C_y>0$ satisfying
 $$
 \nu(\Phi_p(y,\delta_\lambda z))\leq C_y\lambda^{\varrho(y)},
$$
for every $\lambda \leq 1$ and $\|z\|=1$.
Thanks to \r{pro},
$$
\Phi_p(y,z)=\Phi_q(0,z),
$$
where $q=\Phi_p(y,0)$. Hence $\Phi_p(y,\delta_\lambda z)=\Phi_q(0,\delta_\lambda z)=\Phi_q(\delta_\lambda(0,z))
$,
and $\varrho(y)$ is simply the nonholonomic order of $\nu$ at $\Phi_p(y,0)$ and, in particular,  $\varrho(y)\geq \varrho_{\min}$ (we omit the dependance w.r.t.\ $(p,N)$). Assume $\varrho_{\min}<\infty$.
Since $q\mapsto\nu^2(q)$ is smooth, there exists $C>0$ such that
 $$
 \max_{ \|y\| \leq R ,\|z\|_p=1}\nu(\Phi_p(y,\delta_\lambda z))\leq C \lambda^{\varrho_{\min}}.
 $$
Thus
 \begin{equation*}
   \int_0^{R}\lambda^{Q(p)-Q_N-1}\,\left(  \int_{\|y\|\leq {R}} \left( \int_{\{\|z\|=1\}} \frac{d\tilde{\sigma}(z)}{\nu(\Phi(y,\delta_\lambda z))} \right)dy  \right) \,d\lambda \geq C \int_0^1\lambda^{Q(p)-Q_N-1-\varrho_{\min}}d\lambda.
 \end{equation*}
 Therefore, applying \eqref{plm} in Lemma~\ref{tech}, condition $\varrho_{\min}\geq Q(p)-Q_N$
gives the non integrability.
If $\varrho_{\min}=+\infty$, then we directly infer the conclusion, as the inequality
$$
 \int_0^{R}\lambda^{Q(p)-Q_N-1}\,\left(  \int_{\|y\|\leq {R}} \left( \int_{\{\|z\|=1\}} \frac{d\tilde{\sigma}(z)}{\nu(\Phi(y,\delta_\lambda z))} \right)dy  \right) \,d\lambda   \geq C\int_0^1\lambda^{Q(p)-Q_N-1-\alpha}d\lambda
$$
holds for every positive integer $\alpha$.
\end{proof}

\brem
\label{ggggg}
For every $q=\Phi_p(y,0)\in N$ and every $\cI\in {\cal F}$, $\ord_q(\omega (\cX_\cI))\geq Q(p)-Q_\reg$ and thus
$$
\varrho_{\min}(p,N) \geq Q(p)-Q_\reg.
$$
To see this, let $\widehat X_1,\dots, \widehat X_m$ be a nilpotent approximation at $q\in N$ of the generating family $X_1,\dots, X_m$.
Denote by $\widehat \cX_\cI$ the family $\widehat X_{I_1},\dots, \widehat X_{I_n}$, where $\cI=(I_1,\dots, I_n)$ and $\widehat X_{I_j}$ is the Lie bracket between $\widehat X_1,\dots, \widehat X_m$.
Then
$$
\omega (\cX_\cI)(\Phi_p(y,\delta_\lambda z))=\omega (\cX_\cI)(\Phi_q(0,\delta_\lambda z))=\lambda^{Q(q)-Q_\reg}\left(\omega (\widehat \cX_\cI)(\Phi_q(0,\delta_\lambda z))+\lambda b_{\cI,q}(\lambda,z)\right),
$$
where $ b_{\cI,q}$ is a smooth function. Finally, the equisingularity of $N$ ensures that $Q(q)=Q(p)$.
\erem
Proposition~\ref{fin}, jointly with Remark~\ref{ggggg} allows to deduce the following criterion for $\ss^{Q_\Reg}\llcorner_\Reg$ not to be Radon.
\bc
\label{th:s>r}
Assume (A). If $\dim_H\Sing\geq \dim_H\Reg$, then  $\ss^{Q_\Reg}\llcorner_\Reg$ is not a Radon measure.
\ec
\brem
 Denoting by $\Sing=\cup\Sing_j$ the stratification of $\Sing$  by equisingular submanifolds,   condition $\dim_H\Sing\geq \dim_H\Reg$ is equivalent to the existence of a stratum $\Sing_j$ such that $Q_{\Sing_j}\geq Q_\Reg$. Indeed, thanks to Theorem~\ref{mainth},
$$
\dim_H\Sing=\max_j\dim_H\Sing_j=\max_jQ_{\Sing_j}.
$$
Therefore, such condition encodes a metric information, in the sense that the Hausdorff dimension is  defined only through  the metric structure of the manifold, and it bears an algebraic interpretation (and hence a direct way to be tested) in terms of $Q_\Reg$ and the $Q_{\Sing_j}$.
\erem

\begin{proof}
Under the assumptions, there exists an equisingular submanifold $N\subset\Sing$ such that $\dim_H N=\dim_H\Sing$. On the other hand, $Q_N=\dim_H N$ by Theorem~\ref{mainth}. Using Remark~\ref{ggggg}, for every $p\in N$ we have
$$
\varrho_{\min}(p,N) \geq Q(p)-Q_\Reg \geq Q(p) - Q_N,
$$
hence the assumption of Proposition~\ref{fin} is satisfied and $\ss^{Q_\Reg}\llcorner_\Reg$ is not locally integrable around $p$.
\end{proof}

%

As concerns  sufficient conditions for local integrability, we use the same approach as in Proposition~\ref{fin} and we look for uniform lower bounds on $\nu$.

For a point $p\in\Sing$ and an equisingular submanifold $N$ containing $p$, we define
$$
\varrho_{\max}(p,N) = \inf\left\{s\geq 0 \, : \, \inf_{y\in\Omega_1, \|z\|_p=1}\left\{\liminf_{\lambda\to 0}\frac{\nu(\Phi_p(y,\delta_\lambda z))}{\lambda^s}\right\}>0\right\},
$$
where $\Phi_p$ and $\Omega_1$ are constructed in Lemma~\ref{cp}.
By construction, $\varrho_{\min}(p,N) \leq \varrho_{\max}(p,N)$.

\brem
The number $\varrho_{\max}=\varrho_{\max}(p,N) $ is defined precisely to get a uniform lower bound on $\nu$. More precisely, if $\varrho_{\max}<\infty$, there exists $C>0$ such that,
for every $(y,z)\in\Omega_1\times\{\|z\|_p=1\}$ and $\lambda\in[0,1]$, there holds
\begin{equation}
\label{dfghj}
\nu(\Phi_p(y,\delta_\lambda z))\geq C\lambda^{\varrho_{\max}}.
\end{equation}
Unlike $\varrho_{\min}$, in general $\varrho_{\max}$
does not bear an algebraic interpretation as a nonholonomic order. Nevertheless, it can be characterized in a more intrinsic way as
$$
\varrho_{\max}=\inf\left\{s\geq 0 \, : \, \inf_{\bar q\in B(p,\eps)\cap N}\left\{\liminf_{q\to\bar q, q\notin \Sing}\frac{\nu(q)}{d(q,N)^s}\right\}>0\right\}.
$$
\erem
As a direct consequence of \r{dfghj} and Lemma~\ref{tech}, we deduce the following fact.
\bp\label{finito}
Assume (A). Let $p\in\Sing$ and let $N$ be an equisingular submanifold containing $p$.
If $\varrho_{\max}(p,N)<Q(p)-Q_N$ then $\ss^{Q_\Reg}\llcorner_{\Reg}$ is locally $\mu$-integrable around $p$.
\ep
Notice that the hypothesis   $\varrho_{\max}(p,N) <Q(p)-Q_N$ implies in particular that $\varrho_{\max}<\infty$, and thus that in a neighborhood of $p$ the set $\Sing$ coincides with $N$.

In general, condition $\varrho_{\max}<Q(p)-Q_N$ is not sharp, in the sense that there may be cases in which the condition fails and $\ss^{Q_\Reg}\llcorner_{\Reg}$ is locally integrable (see Example~\ref{ex:last}).  
Actually, the study of the generic case (see the proof of Proposition~\ref{gensmo}) shows that the integrability does not necessarily
 require a bound on $\varrho_{\max}$. Nevertheless, when $\nu$ is homogeneous, we have $\varrho_{\max}=\varrho_{\min}$ and thus Proposition~\ref{fin} together with Proposition~\ref{finito} provides a characterization of integrability.  Namely, if $\nu$ is homogeneous, then $\ss^{Q_\Reg}_\Reg$ is locally integrable at $p\in N$ if and only if $\varrho_{\min}<Q(p)-Q_N$.

\subsubsection{Conditions on usual orders}

Consider a singular point $p\in\Sing$ and  an equisingular submanifold $N \subset \Sing$ containing $p$. Define
$$
\ee_{\min}(p,N)=\lim_{\eps\to 0}\min\{\mathrm{ord}^{\mathrm{diff}}_{q}(\omega(\cX_\cI))\mid {q\in N\cap B(p,\eps)}, \ \cI\in\cF \},
$$
where $\mathrm{ord}^{\mathrm{diff}}_{q}$ is the usual notion of order of the differential calculus. In coordinates $x$, the order $\mathrm{ord}^{\mathrm{diff}}_{q} h$ of a smooth function $h$ is the biggest integer $k$ such that all partial derivatives of $h$ of order smaller than $k$ are zero. Equivalently, it is the biggest integer such that locally $h(x) = O(|x|^k)$. Obviously,
$$
\ee_{\min}(p,N) \leq \varrho_{\min}(p,N).
$$

\bp
\label{le:domega=0}
Assume (A). Let $p\in\Sing$ and let $N$ be an equisingular submanifold of codimension $n-k\geq 1$. If $\ee_{\min}(p,N) \geq n-k$,
then $\frac{d\ss^{Q_\Reg}\llcorner_{\Reg}}{d\mu}$ is not locally $\mu$-integrable around $p$.
\ep

\begin{proof}
The proof is similar to the one of Proposition~\ref{fin}, replacing $\ord_q$ by $\mathrm{ord}^{\mathrm{diff}}_{q}$ and \eqref{plm} by \eqref{eq:eucl}.
\end{proof}

\bc
\label{tre}
Assume (A). If $\Sing$ contains a submanifold  of codimension 1, then $\ss^{Q_\Reg}\llcorner_\Reg$ is not a Radon measure.
\ec

\begin{proof}
 By assumption (A), $\Sing$ is stratified by equisingular submanifolds. Therefore  there exists an equisingular submanifold $N\subset \Sing$ of codimension one.
   Let $p\in N$.  Take a coordinate system $(y_1,\dots, y_{n-1},z)$ that allows to identify $N$ with $\{z=0\}$.
For every $\cI\in\cF$, since $\omega(\cX_{\cI})(\Phi_p(y,z))\equiv 0$, it results from  Malgrange Preparation Theorem that there exists $C>0$ such that
$$
|\omega(\cX_{\cI})(\Phi_p(y,z))|\leq C|z|,
$$
  for every $y$ in a neighborhood of $0$, and thus $\ee_{\min}(p,N) \geq 1=n-(n-1)$. Proposition~\ref{le:domega=0} applies.
\end{proof}

A sufficient condition of integrability similar to Proposition~\ref{finito} can also be derived.

%

\section{Hausdorff measures on submanifolds}\label{seqsm}
 Thanks to the analysis  in Section~\ref{se:stratif}, under assumption (A) the study of the singular part $\volh\llcorner_\Sing$ reduces to the study of the restriction of this measure to every equisingular stratum $\Sing_i\subset\Sing$. Since $\volh\llcorner_{\Sing_i}=0$ or $+\infty$ if $\dim_H \Sing_i \neq \dim_H M$, we are actually led to study $\ss^{\dim_H N}\llcorner_{N}$ for an equisingular submanifold $N$. This is the object of this section.

We  first describe the algebraic structure of an equisingular manifold $N$, and then give  in Theorem~\ref{mainth} a complete description of $\ss^{\dim_H N}\llcorner_{N}$. Finally we give an estimate of the Radon-Nikodym derivative of $\ss^{\dim_H N}\llcorner_{N}$ with respect to a smooth  volume  on $N$.


\subsection{Algebraic structure associated with equisingular submanifolds}

 Recall that the metric tangent cone to $(M,d)$ at any point $p$ exists and it is isometric to $(T_pM,\widehat d_p)$, where  $\widehat d_p$ denotes the sub-Riemannian distance associated with a nilpotent approximation at $p$ (see \cite{bellaiche}).

The following proposition shows the relevance of equisingular submanifolds as particular subsets of $(M,d)$ for which a metric tangent cone  exists. Such metric space is isometrically embedded in a metric tangent cone to the whole $M$ at the point.

\bp\label{densities}
Let $N\subset M$ be an  equisingular submanifold of $M$ of dimension $k$ and $K\subset N$ be a compact set.
The following properties hold.
\bi
\item[$\mathrm{(i)}$]  There exists a metric tangent cone to
  $(N,d|_N)$ at every $p \in N$ and it is isometric to $(T_pN,\widehat
  d_p|_{T_pN})$.

\item[$\mathrm{(ii)}$]  For every $p\in N$, the graded vector space
 $$
 \mathfrak{gr}_p^N(\bD):= \bigoplus_{i=1}^{r(p)}(\bD^{i}_p\cap T_pN)/(\bD^{i-1}_p\cap T_pN)
 $$
 is a nilpotent Lie algebra whose associated Lie group $\mathrm{Gr}_p^N(\bD)$ is diffeomorphic to $T_pN$.

 \iii[$\mathrm{(iii)}$] Denote by $\widehat B_p$ the unit ball for $\widehat d_p$. The function $p\mapsto \diam\,_{\widehat d_p}(T_pN\cap \widehat B_p)$ is continuous on $N$ and
  \begin{equation}
  \label{limdiam}
  \sup_{p\in K}\left|\frac{1}{\eps}\diam(N\cap B(p,\eps))-\diam\,_{\widehat d_p}(T_pN\cap \widehat B_p)\right|\xrightarrow[\eps\to0^+]{} 0;
 \end{equation}
 \iii[$\mathrm{(iv)}$]  Assume $N$ is oriented and let $\mu$ be a smooth volume on $N$ associated with a $k$-form $\varpi\in\bigwedge^k N$. Then $\varpi$ induces canonically a  left-invariant $k$-form $\wp$ on
 $\mathrm{Gr}_p^N(\bD)$  and a corresponding smooth volume $\mup$. Moreover, the function $p\mapsto \mup (T_pN\cap \widehat B_p)$ is continuous on $N$, and
\begin{equation}
\label{limmu}
\sup_{p\in K}\left|\frac{1}{\eps^{Q_N}}\mu(N\cap B(p,\eps)) -  \mup (T_pN\cap \widehat B_p)\right|\xrightarrow[\eps\to0^+]{} 0.
\end{equation}
\ei
\ep
\brem
When $N$ is an open submanifold of $M$, assuming $N$
equisingular is equivalent to saying that $N$ contains only regular
points. In that case,  most of Proposition~\ref{densities} is well-known: point (i)
follows by the fact that  the nilpotent approximation is a metric
tangent cone; point (ii) says that the tangent cone shares a group
structure; point (iii) is trivial since $\diam\,_{\widehat d_p}(\widehat B_p)=2$ and $\diam(B(p,\eps)) = 2 \eps$ for small $\eps$; and the continuity of $p\mapsto \mup (\widehat B_p)$ in (iv) has been
remarked in \cite{balu}. Only the uniform convergence \r{limmu} is new in this case.
\erem

\begin{proof}
Note first that since the result is of local nature, it is sufficient
that we prove it on a small neighborhood $U \cap N$ of a point
$p_0 \in N$. It results from Lemma~\ref{cp} that, for every $p$ in a such a neighborhood, there exists a
coordinate system
$\varphi_p:U_p\to \R^n$ on a neighborhood $U_p \subset M$ of $p$, such that
 $\varphi_p$ are privileged coordinates at $p$,  $p \mapsto \varphi_p$ is smooth on $U \cap N$, and $N$ is rectified in coordinates $\varphi_p$, that is $\varphi_p(N \cap U_p) \subset \{x\in\R^n\mid x_{k+1}=\dots=x_n=0\}$.\smallskip


\emph{Proof of (i) and (ii).} Coordinates $\varphi_p$ allow to identify $(T_pM,\widehat d_p)$ with $(\R^n,\widehat d_p)$. From \cite[Theorem~7.32]{bellaiche} we have the following estimate, which is instrumental in the proof below. There exists two positive constants $\eps_p$, $C_p$ and an integer $r$ such that, if $x,y \in \R^n$ satisfy $M_p(x,y):=\max (\widehat d_p(0,x) ,\widehat d_p(0,y)) <\eps_p$, then
\begin{equation}
\label{eq:70}
-C_p M_p(x,y) d(\varphi_p^{-1}(x),\varphi_p^{-1}(y))^{1/r} \leq d(\varphi_p^{-1}(x),\varphi_p^{-1}(y)) - \widehat d_p(x,y) \leq C_p M_p(x,y) \widehat d_p(x,y)^{1/r}.
\end{equation}
Moreover, as noticed in \cite[Sect.\ 2.2.2]{Jea-2014}, since $N$ is equisingular the functions $p \mapsto \eps_p$ and $p \mapsto C_p$ are continuous on $N$ and $r$ does not depend on $p \in N$ ($r$ is the maximum of the weights $w_1(p), \dots, w_n(p)$).

Now, using $\varphi_p$ we identify locally $M$ with $T_pM\simeq\R^n$ and $N$ with  $T_pN \simeq \R^k\times\{0\}$.
Therefore, whenever $q_1,q_2\in U_p\cap N$ we have
$$
\widehat d_p(\varphi_p^{-1}(q_1),\varphi_p^{-1}(q_2))=\widehat d_p|_{T_pN}(\varphi_p^{-1}(q_1),\varphi_p^{-1}(q_2)),
$$
and obviously $d(q_1,q_2)= d|_{N}(q_1,q_2)$. Hence \eqref{eq:70} holds when we restrict $d$ to $N$ and
$\widehat d$ to $T_pN$. This allows to conclude~\cite[Proposition~8.12]{montgomery} that a metric tangent
cone to $(N,d|_N)$ at $p$ exists and it is isometric to
$(T_pN,\widehat d_p|_{T_pN})$.
The algebraic structure of $\mathfrak{gr}_p^N(\bD)$ and the  fact that
$\mathrm{Gr}_p^N(\bD)$ is diffeomorphic to $\R^k$ are obtained in exactly the same way as at regular point, see~\cite[Sect.\ 4.4]{montgomery}.
Thus (i) and (ii) are proved.\smallskip

\emph{Proof of (iii).}
Let $w_i(p)$ be the weights at $p$ labeled according to the coordinates $\varphi_p$. Note that the dilations $\delta_s: x \mapsto (s^{w_1(p)}x_1, \dots, s^{w_n(p)}x_n)$, $s>0$, do not depend on $p\in N$ as $N$ is  equisingular. Thus the homogeneity of $\widehat d_p$ with respect to $\delta_s$ and estimate~\eqref{eq:70} imply that the function $(p,x,y) \mapsto  \widehat d_p(x,y)$ is continuous on $N \times \R^n \times \R^n$. As a consequence,
$$
\diam\,_{\widehat d_p}(T_pN\cap \widehat B_p) = \sup \left\{ \widehat d_p(x,y) \mid x,y \in \R^k, \ \widehat d_p(0,x) \hbox{ and } \widehat d_p(0,y) \leq 1 \right\}
$$
is an upper semi-continuous function of $p$. Let us prove that is also lower semi-continuous. Fix $p\in N$ and let $(x,y) \in \R^k \times \R^k$ such that $\diam\,_{\widehat d_p}(T_pN\cap \widehat B_p) = \widehat d_p(x,y)$. For $q \in N$ we set
$$
\alpha (q) = \min \left( \frac{\widehat d_p(0,x)}{\widehat d_q(0,x)}, \frac{\widehat d_p(0,y)}{\widehat d_q(0,y)} \right).
$$
By construction we have $\widehat d_q(0,\delta_{\alpha(q)}x)$ and $\widehat d_q(0,\delta_{\alpha(q)}y) \leq 1$, hence
$$
\diam\,_{\widehat d_q}(T_qN\cap \widehat B_q) \geq \widehat d_q(\delta_{\alpha(q)}x,\delta_{\alpha(q)}y) = \alpha(q) \widehat d_q (x,y).
$$
When $q \to p$, the continuity of $q \mapsto  \widehat d_q(x,y)$ implies that $\alpha (q) \to 1$ and $\widehat d_q (x,y) \to \widehat d_p (x,y)$. As a result,
$$
\liminf_{q \to p} \diam\,_{\widehat d_q}(T_qN\cap \widehat B_q) \geq  \widehat d_p (x,y)=\diam\,_{\widehat d_p}(T_pN\cap \widehat B_p),
$$
that is, the function $p \mapsto \diam\,_{\widehat d_p}(T_pN\cap \widehat B_p)$ is lower semi-continuous and then continuous.\smallskip

To prove \eqref{limdiam}, we write
$$
\diam (N \cap B(p,\eps)) = \sup \left\{ d(\varphi_p^{-1}(x),\varphi_p^{-1}(y)) \ \mid \ x,y \in \varphi_p \left( B(p,\eps) \right) \cap \R^k \right\}.
$$
From the continuity of the constants in \eqref{eq:70} we deduce the existence of a constant $C>0$ and a function $\rho:\R \to \R$, $\rho(\eps) \to 0$ as $\eps \to 0$, which satisfy the following property: for every $p\in K$,
\begin{equation}
\label{eq:boxunif}
\bbrr (0, \eps (1-\rho(\eps))) \subset \varphi_p \left( B(p,\eps) \right) \subset
\bbrr (0, \eps (1+\rho(\eps))),
\end{equation}
and, if $x,y$ belong to $\bbrr(0, \eps (1+\rho(\eps)))$, then
\begin{equation*}
\label{eq:dhd}
| d(\varphi_p^{-1}(x),\varphi_p^{-1}(y)) - \widehat d_p(x,y) | \leq C (\eps (1+\rho(\eps)))^{1+1/r}.
\end{equation*}
As a consequence, the diameter of $N \cap B(p,\eps)$ satisfies
\begin{multline*}
    \diam\,_{\widehat d_p}\left(T_pN \cap\bbrr (0, \eps (1-\rho(\eps)))\right) - C (\eps (1+\rho(\eps)))^{1+1/r} \  \leq \ \diam (N \cap B(p,\eps))   \\
    \leq \ \diam\,_{\widehat d_p}\left(T_pN \cap\bbrr (0, \eps (1+\rho(\eps)))\right) + C (\eps (1+\rho(\eps)))^{1+1/r},
\end{multline*}
for every $p \in K$. From the homogeneity of $\widehat d_p$,
$$
\diam\,_{\widehat d_p}\left(T_pN \cap \bbrr (0, \eps (1 \pm \rho(\eps)))\right) = \eps (1\pm \rho(\eps)) \diam\,_{\widehat d_p}\left(T_pN \cap \widehat B_p\right),
$$
and we obtain \eqref{limdiam} from the inequalities above. Thus point (iii) is proved.\smallskip

\emph{Proof of (iv).}  By  (ii)  there  exists a canonical isomorphism $\sigma$ between
 $\bigwedge^k(T^*_pN)$ and $\bigwedge^k(\mathfrak{gr}^N_p(\bD)^*)$ (see the construction in \cite[Section~10.5]{montgomery}). Let
$\sigma (\varpi_p)$ be the image under such isomorphism of the value $\varpi_p$ of $\varpi$ at $p$. Then $\wp$
is defined as the left-invariant $k$-form on $\mathrm{Gr}_p^N(\bD)$ which coincides
with $\sigma (\varpi_p)$ at the origin.

 The next step is to show that the function $p\mapsto \mup (T_pN\cap \widehat B_p)$ is continuous on $N$. In coordinates $\varphi_p$, we identify $T_pN$ and $\mathrm{Gr}_p^N(\bD)$ with $\R^k \times \{0\} \subset \R^n$. Through this identification, $\sigma (\varpi_p)$ coincides with $ \varpi_p$ and may be  written as  $\varpi_p(\partial_{x_1},\dots,\partial_{x_k}) (dx_1 \wedge \cdots \wedge dx_k)_p$.
As a consequence
$$
\wp =\varpi_p(\partial_{x_1},\dots,\partial_{x_k}) \, dx_1 \wedge \cdots \wedge dx_k \quad \hbox{and} \quad
\mup= \varpi_p(\partial_{x_1},\dots,\partial_{x_k}) \Leb{k} ,
$$
where $\Leb{k}$ is the Lebesgue measure on $\R^k$.

 Continuity of $p\mapsto \mup (T_pN\cap \widehat B_p)$ on $N$ reduces to continuity of  $p\mapsto \Leb{k} (T_pN\cap \widehat B_p)$ on $N$.  As for the proof of the continuity of $p \mapsto \diam\,_{\widehat d_p}(T_pN\cap \widehat B_p)$, the  crucial point  is the homogeneity of $\widehat d_p$. Fix $p\in N$ and define the functions $\alpha, \beta : N \mapsto\R$ as follows. For $q \in N$,
$$
\alpha (q) = \max \{ \widehat d_q (0,x) \ \mid \ x \in \R^k, \ \widehat d_p (0,x) =1\}, \qquad \beta (q) = \min \{ \widehat d_q (0,x) \ \mid \ x \in \R^k, \ \widehat d_p (0,x) =1\}.
$$
The continuity of $q \mapsto  \widehat d_q$ implies that $\alpha$  and $\beta$  are continuous. Moreover, by homogeneity of the distances $\widehat d_q$, we have $\delta_{\beta(q)}\widehat B_q\, \cap \,\R^k \subset \widehat B_p \,\cap \,\R^k \subset \delta_{\alpha(q)}\widehat B_q \,\cap \,\R^k$. Recall that, for $x \in \R^k$  $\delta_s x = (s^{w_1(p)}x_1, \dots, s^{w_k(p)}x_k)$  and that $Q_N = \sum_{i=1}^k w_i(q)$ for every $q\in N$ (see remark~\ref{re:QN}).
Hence
$$
\frac{1}{\alpha(q)^{Q_N}} \Leb{k} (\widehat B_p \cap \R^k) \leq \Leb{k} (\widehat B_q \cap \R^k) \leq \frac{1}{\beta(q)^{Q_N}} \Leb{k} (\widehat B_p \cap \R^k).
$$
Since  $\alpha$ and $\beta$ are continuous  and converge to $1$ as $q\to p$,   $\Leb{k} (\widehat B_q \cap \R^k)$ converges to $\Leb{k} (\widehat B_p \cap \R^k)$ as $q \to p$. This proves the continuity of $p\mapsto \Leb{k} (T_pN\cap \widehat B_p)$ and so the one of $p\mapsto \mup (T_pN\cap \widehat B_p)$.\smallskip

It remains to prove \eqref{limmu}. Let us write the measure $\mu$ in coordinates $\varphi_p$,
$$
\mu(N\cap B(p,\eps)) = \int_{\varphi_p(N\cap B(p,\eps))} {\varphi_p}_*\varpi(\partial_{x_1},\dots,\partial_{x_k})(x) \, dx_1 \wedge \cdots \wedge dx_k .
$$
Using \eqref{eq:boxunif} and the smoothness of the function $(p,x) \mapsto {\varphi_p}_*\varpi(\partial_{x_1},\dots,\partial_{x_k})(x)$, we obtain the following inequalities,
\begin{multline*}
\int_{T_pN \cap \bbrr (0, \eps (1-\rho(\eps)))} \Big(\varpi_p(\partial_{x_1},\dots,\partial_{x_k})- \rho(\eps) \Big) \, dx_1 \wedge \cdots \wedge dx_k \ \leq \ \mu(N\cap B(p,\eps)) \\
\leq \ \int_{T_pN \cap\bbrr (0, \eps (1+\rho(\eps)))} \Big(\varpi_p(\partial_{x_1},\dots,\partial_{x_k}) + \rho(\eps) \Big) \, dx_1 \wedge \cdots \wedge dx_k.
\end{multline*}
The conclusion follows from
\begin{equation}
\label{eq:mup}
\mup (T_pN\cap \widehat B_p) = \frac{1}{\eps^{Q_N}} \mup (T_pN\cap \bbrr (0, \eps )) = \frac{1}{\eps^{Q_N}}\int_{T_pN\cap \bbrr (0, \eps )} \varpi_p(\partial_{x_1},\dots,\partial_{x_k}) \, dx_1 \wedge \cdots \wedge dx_k .
\end{equation}
\end{proof}

\subsection{Hausdorff volume of equisingular submanifolds}\label{hjkl}

 Recall that $\volh^N=\ss^{\dim_HN}_N$ (see Section~\ref{hhss}).

\bt\label{mainth}
Let $N\subset M$ be an oriented equisingular submanifold and let $\mu$ be a smooth volume on $N$. Then
\begin{enumerate}
\iii $\dim_{H}N=\qn$ and $\volh^N=\ss^{\qn}_N$;
\iii $\volh^N$ is a Radon measure on $N$, i.e, $\volh^N(K)<\infty$ for every compact set $K\subset N$;
\iii $\volh^N\ll\mu$ and $\mu\ll\volh^N$;
\iii the Radon-Nikodym derivative of $\volh^N$ with respect to $\mu$ is the density $\lim_{\eps\to 0}\frac{\volh^N(N\cap B(p,\eps))}{\mu(N \cap
  B(p,\eps))}$, whose value is
    \begin{equation}\label{abscont}
\lim_{\eps\to 0}\frac{\volh^N(N\cap B(p,\eps))}{\mu(N \cap
  B(p,\eps))}=\frac{\diam_{\widehat d_p}(T_pN\cap \widehat
  B_p)^{Q_N}}{\mup(T_pN\cap \widehat B_p)},~~ \forall\, p\in N.
\end{equation}
\end{enumerate}
\et

\brem
When $M$ is equiregular and connected, applying Theorem~\ref{mainth} to $N=M$ we deduce Theorem~\ref{mainthreg}.
\erem

\brem
Point 4 together with points (iii) and (iv) of Proposition~\ref{densities} shows that the Radon-Nikodym derivative $\frac{d\volh^N}{d\mu}(q)$ is continuous on $N$.
\erem

Before starting the proof of the theorem, we need to establish the following simple lemma.

\bl
\label{le:ss_vs_mu}
Let $N$ and $\mu$ be as in Theorem~\ref{mainth}. Let $p\in N$. Assume   there exists positive constants $\eps_0$ and $\mu_+ > \mu_-$ such that, for every $\eps<\eps_0$ and every point $q \in N\cap B(p,\eps_0)$, there holds
\begin{equation}
\label{eq:mu+-}
    \mu_- \diam(N\cap B(q,\eps))^{Q_N} \leq \mu (N\cap B(q,\eps)) \leq \mu_+ \diam( N\cap B(q,\eps))^{Q_N}.
\end{equation}
Then, for every $\eps<\eps_0$,
\begin{equation}\label{lem2}
\frac{\mu (N\cap B(p,\eps))}{\mu_+}  \leq \ss_N^{Q_N}(B(p,\eps)) \leq  \frac{\mu (N\cap B(p,\eps))}{\mu_-}.
\end{equation}
\el

\begin{proof}
Let $\bigcup_i B(q_i,r_i)$ be a covering of $N\cap B(p,\eps)$ with
balls centered at points in $N$  of radius smaller than $\delta <\eps_0$. If $\delta$ is small enough,
every $q_i$ belongs to $N\cap B(p,\eps_0)$ and, using~\eqref{eq:mu+-}, there holds
$$
 \mu(N\cap B(p,\eps)) \leq \sum_i \mu (N\cap B(q_i,r_i)) \leq \mu_+
\sum_i \diam (N\cap B(q_i,r_i))^{Q_N}.
$$
Hence, we have $\ss_N^{Q_N} (B(p,\eps)) \geq \frac{\mu(N\cap B(p,\eps))}{\mu_+}$.

For the other inequality, let $\eta>0$, $0<\delta <\eps_0 $ and let $\bigcup_i B(q_i,r_i)$ be a covering of $N\cap B(p,\eps)$ such that $q_i\in N\cap  B(p,\eps)$, $r_i <\delta$ and $\sum_i \mu (N\cap B(q_i,r_i)) \leq \mu(N\cap B(p,\eps)) + \eta$. Such a covering exists due to the Vitali covering lemma. Using as above~\eqref{eq:mu+-}, we obtain
$$
\mu(N\cap B(p,\eps)) + \eta \geq \sum_i \mu (N\cap B(q_i,r_i)) \geq \mu_-
\sum_i \diam(N\cap B(q_i,r_i))^{Q_N}.
$$
We then have $\ss_{N,\delta}^\qn (B(p,\eps))  \leq  \frac{N\cap \mu(B(p,\eps))}{\mu_-}  + \frac{\eta}{\mu_-} $. Letting  $\eta$ and $\delta$ tend to $0$, we get the conclusion.
\end{proof}

\begin{proof}[Proof of Theorem~\ref{mainth}]
Let $p\in N$ and fix $\delta>0$.  Set
$$
\mu_\pm(\delta)=\frac{\hat\mu^{p}(T_pN\cap\widehat B_p)}{\diam_{\widehat d_p}(T_pN\cap\widehat B_p)}\pm\delta
$$
Let $K$ be the intersection of $N$ with the closed ball of center $p$ and radius $\delta$. By Proposition~\ref{densities}, the functions $q\mapsto \diam_{\widehat d_q}(T_qN\cap \widehat B_q)$ and $q\mapsto\muq(T_qN\cap \widehat B_q)$ are uniformly continuous on $K$ and  convergences \r{limdiam} and \r{limmu} are uniform on $K$. Therefore there exists $\eps_0>0$ (depending only on  $\delta$)
such that $\eps_0<\delta$ and,  for every $\eps<\eps_0$, for every $q\in N\cap B(p,\eps)$, there holds
$$
    \mu_- (\delta)\diam(B(q,\eps)\cap N)^{Q_N} \leq \mu (B(q,\eps)\cap N) \leq \mu_+(\delta) \diam(B(q,\eps)\cap N)^{Q_N}.
$$
Applying Lemma~\ref{le:ss_vs_mu} we deduce that for every $\eps<\eps_0$ there holds
\begin{equation}\label{tutto}
\mu_-(\delta)  \leq \frac{\ss_N^{Q_N}(B(p,\eps))}{\mu (B(p,\eps)\cap N)} \leq  \mu_+(\delta).
\end{equation}
Therefore we infer that $\ss^{Q_N}_N$ is a Radon measure and, using a covering argument, that $\ss^{Q_N}$ and $\mu$ are mutually absolutely continuous, i.e., property 3.
As $\delta$ goes to $0$ we obtain that
\begin{equation}\label{rrr}
\lim_{\eps\downarrow 0}\frac{\ss_N^{Q_N}(B(p,\eps))}{\mu (B(p,\eps)\cap N)}=\frac{\diam_{\widehat d_p}(T_pN\cap\widehat B_p)}{\hat\mu^{p}(T_pN\cap\widehat B_p)}.
\end{equation}
Since the right-hand side of \r{rrr} is positive and continuous on $N$, for $\eps>0$ small enough,
$$
0<\ss^{Q_N}(N\cap B(p,\eps))<\infty,
$$
whence $\dim_HN=Q_N$ and $\vol^N_H=\ss^{Q_N}_N$. As a consequence, \r{rrr} gives
\r{abscont}.
Finally, we apply the differentiation theorem for Radon measures \cite[Theorem~4.7]{simon} (with $X=N$, $\mu_1=\mu$ and $\mu_2=\vol^N_H$) to obtain that $\lim_{\eps\to 0}\frac{\volh^N(N\cap B(p,\eps))}{\mu(N \cap
  B(p,\eps))}$ coincides with the Radon-Nikodym derivative of $\vol_H^N$ with respect to $\mu$.
\end{proof}


\subsection{Weak equivalent of $q\mapsto \muq(T_qN\cap \widehat B_q)$}\label{fghjk}
We end this section by stating a result which gives a weak
  equivalent of the function $q\mapsto \muq(T_qN\cap \widehat B_q)$ appearing
  in Theorem~\ref{mainth}. This is  instrumental   to
  determine whether the Radon-Nikodym derivative of $\vol_H^N$ with respect to $\mu$ is integrable or
  not. This result stems from the uniform Ball-Box Theorem,
  \cite{jean1} and \cite[Theorem~4.7]{Jea-2014}. A related formula computing explicitly  Popp's measure in equiregular manifolds has been given in \cite{br}.

\begin{proposition}
\label{th:nuq}
Let $M$ be an oriented manifold with a volume form $\omega$, and
 $N \subset M$ be an oriented submanifold  with a volume form $\varpi$. We denote by $\mu$ the associated smooth volume on $N$ and we set $k = \dim N$.
  Finally, let $X_1,\dots, X_{m}$ be a generating family for a sub-Riemannian structure on $M$.

  If $N$ is equisingular, then for any compact subset $K \subset M$ there exists a constant
$C>0$ such that, for every $q \in N \cap K$,
$$
\muq(T_qN\cap \widehat B_q) \asymp_C {\bar \nu(q)},
$$
where ${\bar \nu(q)} = \max\{ \big| \left(\varpi\wedge dX_{I_{k+1}} \wedge \cdots
  \wedge dX_{I_{n}}\right)_q (X_{I_{1}}(q),\dots,X_{I_{n}}(q)) \big| \}$, the
maximum being taken among all $n$-tuples  $(X_{I_{1}},\dots,X_{I_{n}})$
in $\arg \max \{ \omega_q(X_{I'_{1}}(q),\dots,X_{I'_{n}}(q)) \mid \
\sum_i |I'_i| = Q(q) \}$.
\end{proposition}

\begin{remark}\label{ggg}
As a particular case of the proposition, if $N$ is an open equisingular submanifold of $M$ and
$\varpi=\omega$,
$$
\muq(\widehat B_q) \asymp_C \max \left\{
|\omega_q(X_{I'_{1}},\dots,X_{I'_{n}})| \ \mid  \ \sum_i |I'_i| = Q (q) \right\},
\quad \hbox{ for every $q \in N \cap K$}.
$$
\end{remark}

\begin{remark}
Proposition~\ref{th:nuq}, together with Theorem~\ref{mainth}, allows to give an estimate of the Radon-Nikodym derivative of $\volh^N$ with respect to $\mu$. Indeed, there exists a constant
$C>0$ such that, for every $q \in K \cap N$,
\begin{equation}\label{eq:asymhh}
 \frac{d\volh^N}{d\mu}(q) \asymp_{C}  \frac{1}{{\bar \nu(q)}} .
\end{equation}
\end{remark}

\begin{proof}
Let $q$ be a point in $N \cap K$, and $(X_{I_{1}},\dots,X_{I_{n}})$ be a $n$-tuple of brackets in $\arg \max \{ \omega_q(X_{I'_{1}}(q),\dots,X_{I'_{n}}(q)) \mid \ \sum_i |I'_i| = Q(q) \}$ such that
$$
{\bar \nu(q)} =  \left(\varpi\wedge dX_{I_{k+1}} \wedge \cdots
  \wedge dX_{I_{n}}\right)_q (X_{I_{1}}(q),\dots,X_{I_{n}}(q)).
$$
Let $x$ be the local coordinates defined by the diffeomorphism
$$
x \mapsto \exp(x_n X_{I_{n}}) \circ \cdots \circ \exp(x_1 X_{I_{1}})(q),
$$
and denote by $\mathrm{Box}_\cX^2(q,\eps)$ the set of points whose coordinates satisfy $|x_i| \leq \eps^{|I_i|}$, $i=1, \dots, n$. From Proposition~\ref{le:unifbbox} (see Appendix), for $\eps$ small enough,
$$
 \mathrm{Box}_\cX^2(q,\eps/C)  \subset B(q,\eps) \subset \mathrm{Box}_\cX^2(q,C\eps) ,
$$
where the positive constant $C$ depends only on  $\omega$, $K$, and $X_1,\dots, X_{m}$.

Let us write $\bar \nu(q)$ as a limit,
$$
\bar \nu(q) = \limsup_{\eps \to 0} \frac{1}{(2\eps)^{Q(q)}} \int_{\mathrm{Box}_\cX^2(q,\eps)} (\varpi\wedge dX_{I_{k+1}} \wedge \cdots
  \wedge dX_{I_{n}})(\partial_{x_1},\dots,\partial_{x_n}) dx_1 \dots dx_n.
$$
The inclusions above imply

\begin{equation}
\label{eq:nuq}
\bar \nu(q) \asymp_C \limsup_{\eps \to 0} \frac{1}{(2\eps)^{Q(q)}} \int_{B(q,\eps)} \varpi\wedge dX_{I_{k+1}} \wedge \cdots
  \wedge dX_{I_{n}}.
\end{equation}

On the other hand, since $N$ is equisingular, we can construct an adapted basis as in Lemma~\ref{cp}: choose vector fields $Y_1, \dots, Y_k$ defined in a neighborhood of $q$ such that    $Y_i$ is tangent to $N$,
 $Y_i \in \span \{ X_{I_j} \mid \ |I_j| \leq \ell(Y_i) \}$ and  $Y_1, \dots, Y_k$ is a
basis adapted to the flag \r{flagq}  restricted to $N$ at $q$.
Note that the vectors $X_{I_{k+1}}(q),\dots, X_{I_{n}}(q)$ in $T_qM$ are transverse to $T_qN$ since $\bar \nu(q) > 0$. Then, up to a rescaling of $Y_1, \dots, Y_k$, we can assume that the $n$-tuple  $\cY=(Y_1,\dots, Y_k,X_{I_{k+1}},\dots, X_{I_{n}})$ satisfies condition~\eqref{eq:normYs} of Proposition~\ref{le:unifbbox}. Denote by $y$ the coordinates defined by
$$
y \mapsto \exp(y_n X_{I_{n}}) \circ \cdots \circ \exp(y_{k+1} X_{I_{k+1}})  \circ \exp(y_k Y_k) \circ \cdots \circ \exp(y_{1}Y_1) (q),
$$
and by $\mathrm{Box}_\cY^2(q,\eps)$ the set of points whose coordinates satisfy $|y_i| \leq \eps^{|I_i|}$, $i=1, \dots, n$. From Proposition~\ref{le:unifbbox}, for $\eps$ small enough,
\begin{equation}
\label{eq:boxyy}
 \mathrm{Box}_\cY^2(q,\eps/C)  \subset B(q,\eps) \subset \mathrm{Box}_\cY^2(q,C\eps).
\end{equation}
Coordinates $y$ are a particular kind of coordinates $\varphi_q$ constructed in Lemma~\ref{cp}. In particular,  the submanifold $N$  lies in  the set $\{y_{k+1} = \cdots = y_n=0\}$ and, with a little abuse of notations,
\begin{equation}
\label{eq: boxy}
\mathrm{Box}_\cY^2(q,\eps/C) \cap \{y_{k+1} = \cdots = y_n=0\} \subset B(q,\eps) \cap N \subset \mathrm{Box}_\cY^2(q,C\eps) \cap \{y_{k+1} = \cdots = y_n=0\}.
\end{equation}

As in the proof of Proposition~\ref{densities}, we use coordinates $y$ to identify locally $M$ with $T_qM$ and $N$ with $T_qN$. In particular, inclusions \eqref{eq: boxy} imply
$$
\mathrm{Box}_\cY^2(q,\eps/C) \cap \R^k \subset T_q N \cap \bbrr(0, \eps ) \subset \mathrm{Box}_\cY^2(q,C\eps) \cap \R^k.
$$
Using this inclusion in~\eqref{eq:mup}, we get, up to increasing $C$,
\begin{eqnarray*}
\muq(T_qN\cap \widehat B_q)  & \asymp_C &  \frac{1}{\eps^{Q_N}}\int_{\mathrm{Box}_\cY^2(q,\eps) \cap \R^k} \varpi_q(\partial_{y_1},\dots,\partial_{y_k}) dy_1 \dots dy_k \\
  & \asymp_C  &  \varpi_q(\partial_{y_1},\dots,\partial_{y_k}).
\end{eqnarray*}
Since $\partial_{y_i}(q) = Y_i(q)$, $i=1,\dots,k$, and since $X_{I_{k+1}}(q),\dots, X_{I_{n}}(q)$ are transverse to $T_qN$, we have
\begin{eqnarray*}
\varpi_q(\partial_{y_1},\dots,\partial_{y_k})  &= & \varpi_q (Y_1(q),\dots,Y_k(q))\\
   &=& (\varpi\wedge dX_{I_{k+1}} \wedge \cdots
  \wedge dX_{I_{n}})_q(Y_1(q),\dots,Y_k(q),X_{I_{k+1}}(q),\dots, X_{I_{n}}(q)).
\end{eqnarray*}
The latter term can be written as a limit,
\begin{multline*}
(\varpi\wedge dX_{I_{k+1}} \wedge \cdots
  \wedge dX_{I_{n}})_q(Y_1(q),\dots,Y_k(q),X_{I_{k+1}}(q),\dots, X_{I_{n}}(q)) \\
 = \limsup_{\eps \to 0} \frac{1}{(2\eps)^{Q(q)}} \int_{\mathrm{Box}_\cY^2(q,\eps)} (\varpi\wedge dX_{I_{k+1}} \wedge \cdots
  \wedge dX_{I_{n}})(\partial_{y_1},\dots,\partial_{y_n}) dy_1 \dots dy_n \\
 \asymp_C \limsup_{\eps \to 0} \frac{1}{(2\eps)^{Q(q)}} \int_{B(q,\eps)} \varpi\wedge dX_{I_{k+1}} \wedge \cdots
  \wedge dX_{I_{n}}, \qquad \qquad
 \end{multline*}
 where we have used \eqref{eq:boxyy}. The conclusion then follows from \eqref{eq:nuq}.
\end{proof}

\section{Applications}\label{appli}

 In this section we use the ideas of Section~\ref{integra} to characterize local integrability of $\dSdm$ in the generic smooth case. As it turns out, this property depends on the placement of the dimension of the manifold with respect to dimensions of free Lie algebras and on the Hausdorff dimension of certain equisingular submanifolds. We end by listing a number of examples illustrating several possible cases.

\subsection{Generic smooth case}\label{gsc}

Let $M$ be a $n$-dimensional smooth oriented manifold and let $m \in \N$. We consider the set $\mathcal{U}_m$ of sub-Riemannian structures $({\bf U},\ps,f)$ of rank $m$ on $M$, i.e. such that $\mathrm{rank}\ \mathbf{U} =m$. Using local generating families we endow this set with the $\con^\infty$-Whitney topology and we say that a sub-Riemannian structure of rank $m$ is generic if it belongs to some residual subset of $\,\mathcal{U}_m$. Thus the analysis of the singular set of  generic $m$-tuples of vector fields given in \cite{ver94} provides the following description of the singular set of generic sub-Riemannian structures.

Introduce first some notations. Let $\cL$ be the free Lie algebra with $m$ generators. We use $\cL^s$ to denote the subspace generated by elements of $\cL$ of length not greater than $s$, and $\tilde n_s$ to denote the dimension of $\cL^s$.
Let $r$ be the integer such that
$$
\tilde n_{r-1}< n \leq \tilde n_r.
$$
Then the singular set $\Sing$ of a generic sub-Riemannian structure admits a locally finite stratification $\Sing=\cup_{i\in\N}\Sing_i$ by equisingular submanifolds and
$$
\min_{i\in \N}\codim\,\Sing_i=\tilde n_r-n+1,
$$
in particular, $\Sing$ is a $\mu$-negligible set (for every smooth volume $\mu$ on $M$).
Moreover, at regular points the growth vector is the maximal one, i.e., $(\tilde n_1,\dots, \tilde n_{r-1},n)$.

\bp
\label{gensmo}
Let $M$ be a $n$-dimensional smooth oriented manifold, $\mu$ a smooth volume on $M$,  and let $({\bf U},\ps,f)$ be a generic sub-Riemannian structure of rank $m$ on $M$.
\bi
\iii[(i)] If $n=\tilde n_r$ then $\dSdm$ is not locally integrable on $M$.
\iii[(ii)] If $n<\tilde n_r$ then $\dSdm$ is locally integrable near every point of strata $\Sing_i$ of minimal   codimension, i.e., such that $\codim\,\Sing_i=\tilde n_r-n+1$.
\ei
\ep
\begin{proof}
The case $n=\tilde n_r$ is a direct consequence of Corollary~\ref{tre}. Then, we consider the case
$n<\tilde n_r$. Let  $N$ be a stratum of minimal codimension, i.e., $\codim\, N=\tilde n_r-n+1\geq 2$. We  set $k=\dim N=n-(\tilde n_r-n+1)$. Let $\omega$ be a non degenerate $n$-form such that $\mu=\int\omega$.
The construction in \cite{Gershkovich1988} allows to characterize $N$ in  the following way.

Denote by $\alpha_1,\dots, \alpha_m$ a set of generators for the free Lie algebra $\cL$. Consider a sequence of multi-indices $\{I_j\}_{j\in \N}$ such that, for every $s\in\N$, the family $\alpha_{I_1},\dots, \alpha_{I_{\tilde n_s}}$ generates $\cL^s$.
Fix a point $p\in N$ and a local generating\footnote{ For instance, taking a local orthonormal frame $\sigma_1,\dots, \sigma_k$ on ${\bf U}$, set $X_i=f\circ \sigma_i$. } family $(X_1, \dots, X_m)$ of $({\bf U},\ps,f)$ near $p$. Possibly reordering multi-indices of length $r$, the vectors $X_{I_1}(p), \dots, X_{I_{n-1}}(p)$ are linearly independent. Then,  one obtains  adapted bases  at regular points near $p$ by taking $X_{I_1},\dots,X_{I_{n-1}}$ and a last vector field chosen among $X_{I_n},\dots, X_{I_{\tilde n_r}}$. There are $n-k=\tilde n_r-n+1$ such bases, which we denote by $\cX_{\cI^{k+1}}, \dots, \cX_{\cI^{n}}$ (for the notation $\cX_{\cI}$ and $\omega(\cX_{\cI})$ see Section~\ref{se:stratif}). Then the map
$$
q\mapsto (\omega(\cX_{\cI^{k+1}})(q),\dots, \omega(\cX_{\cI^{n}})(q))
$$
is a submersion at $p$ and, locally,
$$
N=\{q\mid  \omega(\cX_{\cI^{k+1}})(q)=\dots=\omega(\cX_{\cI^{n}})(q)=0 \}.
$$

Note that  the sub-Riemannian manifold  satisfies assumption (A). By Proposition~\ref{we}, it suffices to show that $1/ \nu$ is locally integrable near $p$. Set $z_i = \omega(\cX_{\cI^{i}})$, $i= k+1, \dots, n$. There exists functions $y_1,\dots, y_k$ near $p$ such that $(y,z)$ is a system of local coordinates near $p$. In these coordinates, $N$ is identified with the set $\{(y,z)\mid z=0\}$ and $\nu(z) = \sqrt{z_{k+1}^2 + \cdots + z_n^2}$. Thus $1/\nu$ is locally integrable near $p$, which ends the proof.

\end{proof}

\subsection{Examples}

In this section we present several examples where assumption (A) is satisfied and one can directly tell whether $\ss^{Q_\Reg}$ is integrable or not using criteria in Section~\ref{integra}.

Recall that a sub-Riemannian structure  $({\bf U},\ps, f)$ on   an $n$-dimensional manifold $M$ is called almost-Riemannian if the rank of ${\bf U}$ is $n$ (see \cite{ABS}).

\begin{example}[Generic almost-Riemannian structures]
\label{gars}
Let $({\bf U},\ps, f)$ be a generic almost-Riemannian structure on $M$.
At regular points the structure is Riemannian, i.e., the growth vector is simply $(n)$ and $Q_{\Reg}=n$.  In particular, $n=\tilde n_1$ and thus the first case of Proposition~\ref{gensmo} applies and $\frac{d \ss^n\llcorner_\Reg}{d \mu}$ is never locally integrable (with respect to a smooth volume) around singular points. As a consequence, $\ss^n\llcorner_\Reg$ is never a Radon measure. Also, an alternative proof of this fact comes as a consequence of Corollary~\ref{th:s>r} or Corollary~\ref{tre}.
\end{example}

We next build examples of almost-Riemannian structures which are not generic and for which $\ss^{n}$ is or fails to be a Radon measure.

\begin{example}[Non generic almost-Riemannian structures]
We consider the almost-Riemannian structure on $\R^3$ for which a global generating family  is
$$
X_1(x_1,x_2,x_3)=\partial_1,~~X_2(x_1,x_2,x_3)=\partial_2,~~X_3(x_1,x_2,x_3)=(x_1^2+x_2^2) \partial_3,
$$
and the canonical volume form $\omega=dx_1 \wedge dx_2 \wedge dx_3$.
The singular set  $\Sing$ coincides with $\{(x_1,x_2,x_3)\mid x_1=x_2=0\}$ and $Q_\Reg= 3$. The growth vector at a singular point $p$ is $(2,2,3)$, $Q(p)=5$, and the growth vector of the flag restricted to $\Sing$ is $(0,0,1)$, whence $Q_\Sing=3$. Hence the Hausdorff dimension of $\R^3$ (endowed with the sub-Riemannian distance) is $3$ and $\volh=\ss^3$.
Therefore conditions of Corollary~\ref{th:s>r} are satisfied and $\frac{d \ss^3\llcorner_\Reg}{d \mu}$ is not integrable near a singular point. In other words, small neighborhoods of singular points have infinite Hausdorff volume.

One could also obtain the conclusion by applying Proposition~\ref{le:domega=0} since the only adapted basis at regular points is $X_1,X_2,X_3$. Indeed,
$$
\nu(x_1,x_2,x_3)\asymp x_1^2+x_2^2,
$$
and $\nu$ is homogeneous of degree $2$. Note that $\varrho_{\min}(p,\Sing)=\varrho_{\max}(p,\Sing)=2=Q(p)-Q_\Sing$.

Consider now the analogue structure in higher dimension, that is, the almost-Riemannian structure on $\R^4$ for which a global generating family is
$$
X_1(x)=\partial_1,~~X_2(x)=\partial_2,~~X_3(x)=\partial_3,~~X_4(x)=(x_1^2+x_2^2+x_3^2) \partial_4,
$$
where $x=(x_1,x_2,x_3,x_4)$, and the canonical volume form $\omega=dx_1\wedge dx_2\wedge dx_3\wedge dx_4.$
The singular set is $\Sing=\{x\in\R^4\mid x_1=x_2=x_3=0\}$ and $Q_\Reg=4$. The growth vector at a singular point $p$ is $(3,3,4)$, $Q(p)=6$, and the growth vector of the flag restricted to $\Sing$ is $(0,0,1)$, whence $Q_\Sing=3$. Hence the Hausdorff dimension of $\R^4$ (endowed with the sub-Riemannian distance) is $4$ and $\volh=\ss^4$. The only adapted basis at regular points is $X_1,X_2,X_3,X_4$, which gives
$$
\nu(x)\asymp x_1^2+x_2^2+x_3^2,
$$
i.e., $\nu$ is homogeneous of degree $2$ and $\varrho_{\min}(p,\Sing)=\varrho_{\max}(p,\Sing)=2<Q(p)-Q_\Sing$.
Then by Proposition~\ref{finito} $\frac{d \ss^4}{d \mu}$ is  integrable  near any singular point. Equivalently, $\ss^4$ is a Radon measure.
\end{example}

\begin{example}[the Martinet space] Recall the sub-Riemannian structure of Example~\ref{mart}, where $M=\R^3$ and a global generating family is
$$
X_1=\partial_1,\quad X_2=\partial_2+\frac{x_1^2}{2}\partial_3.
$$
The singular set is the plane $\Sing=\{x\in\R^3\mid x_1=0\}$ which is equisingular and has codimension 1. The growth vector at any singular point $p$ is $(2,2,3)$, $Q(p)=5$ and the growth vector restricted to $\Sing$ is  $(1,1,1)$, whence $Q_\Sing=4$. Since $Q_\Reg=4$, the Hausdorff dimension of the Martinet space is $4$ and the Hausdorff volume is $\volh=\ss^4$. The only adapted basis at regular points is $X_1,X_2,[X_1,X_2]$, whence
$\nu(x)\asymp |x_1|$, in particular $\nu$ is homogeneous and $\varrho_{\min}(p,\Sing)=\varrho_{\max}(p,\Sing)=1$. Either integrating directly or applying Corollary~\ref{tre} one infers that  $\frac{d \ss^4 \llcorner_\Reg}{d \mu}$   is not integrable near singular points.
\end{example}

Recall that the parameter  $\varrho_{\min}$ used to characterize non integrability admits a lower bound, namely $\varrho_{\min}\geq Q(p)-Q_\Reg$ (see Remark~\ref{ggggg}).
We next see an example showing that $\varrho_{\min}$  can be  greater than this bound.
\begin{example}
Consider the sub-Riemannian structure on  $\R^5$ for which a global generating family is
$$
X_1= \partial_1, \quad X_2 = \partial_2 + x_1 \partial_3 + x_1^2 \partial_5, \quad X_3 = \partial_ 4 + (x_1^k + x_2^k) \partial_5,
$$
with $k\in\N$,
and the canonical volume form $\omega=dx_1\wedge dx_2\wedge dx_3\wedge dx_4\wedge dx_5$.
Computing Lie brackets one obtains
$$
X_{12}=[X_1,X_2]=\partial_3+2x_1\partial_5,~~X_{13}=kx_1^{k-1}\partial_5, ~X_{23}=kx_2^{k-1}\partial_5,~X_{112}=2\partial_5.
$$
Hence at regular points the growth vector is $(3,5)$, $Q_\Reg=7$, and the singular set is $\Sing  = \{x_1 = x_2= 0\}$. At a  singular point $p$ the growth vector is $(3,4,5)$,  $Q(p)=8$ and the growth vector of the flag restricted to $\Sing$ is $(1,2,3)$, which gives $Q_\Sing=6$. Hence the Hausdorff dimension of $\R^5$ (endowed with the sub-Riemannian distance) is $7$ and the Hausdorff volume is $\volh=\ss^7$. At regular points there are two adapted bases $X_1,X_2,X_3,X_{12},X_{13}$ and $X_1,X_2,X_3,X_{12},X_{23}$, which implies that
$$
\nu(x_1,x_2,x_3,x_4,x_5)\asymp \sqrt{x_1^{2(k-1)}+x_2^{2(k-1)}}.
$$
Thus $\nu$ is homogeneous of order $k-1$ and $\varrho_{\min}(p,\Sing)=\varrho_{\max}(p,\Sing)=k-1$. In particular $\varrho_{\min}(p,\Sing)>Q(p)-Q_\Reg$ if and only if $k>2$. Applying Proposition~\ref{fin} and Proposition~\ref{finito}, $\frac{d \ss^7}{d \mu}$ is locally integrable around singular points if and only if $k\leq 2$.
\end{example}

We end by illustrating a case where  $\varrho_{\min}<Q(p)-Q_\Sing<\varrho_{\max}<\infty$
 and $\volh$ is locally integrable, showing that the condition on $\varrho_{\max}$ given in Proposition~\ref{finito} is not sharp.

\begin{example}
\label{ex:last}
Consider the sub-Riemannian structure on $\R^4$ for which a global orthonormal frame is
$$
X_1(x)=\partial_1,~X_2(x)=\partial_2+x_1\partial_3+(x_1^2x_3^2-x_1x_2^2)\partial_4,
$$
and the canonical volume form $\omega=dx_1\wedge dx_2\wedge dx_3\wedge dx_4$. Computing Lie brackets one obtains
$$
X_{12}(x)=\partial_3+(2x_1x_3^2-x_2^2)\partial_4, \quad
X_{112}(x)=2x_3^2\partial_4,
\quad X_{212}(x)=(4x_1^2x_3-2x_2)\partial_4.
$$
Hence, the singular set is $\Sing=\{x\in\R^4\mid x_2=x_3=0\}$,  at regular points the growth vector is $(2,3,4)$ and $Q_\Reg=7$. At a singular point $p$  the growth vector is $(2,3,3,4)$, $Q(p)=8$, and the growth vector of the flag restricted to $\Sing$ is $(1,1,1,2)$, which gives $Q_\Sing=5$. Thus the Hausdorff dimension of $\R^4$ (endowed with the sub-Riemannian distance) is $7$ and the Hausdorff volume is $\volh=\ss^7$. At regular points there are two adapted bases $X_1,X_2,X_{12},X_{112}$ and $X_1,X_2,X_{12},X_{212}$, therefore
$$
\nu(x)\asymp \sqrt{\omega_1(x)^2+\omega_2(x)^2},
$$
where $\omega_1(x)=\det(X_1,X_2,X_{12},X_{112})=2x_3^2$ and $\omega_2(x)=\det(X_1,X_2,X_{12},X_{212})=3x_1^2x_3^2-2x_2$.
According to the notation in Proposition~\ref{tech}, coordinates  $y=(x_1,x_4)$ (of weights $1,4$ at $p=0$ respectively) parameterize the singular set, whereas coordinates $z=(x_2,x_3)$ (of weights $1,2$ at $p$ respectively) are transversal to $\Sing$.
Apply   the coordinate change  $(y,z)\mapsto(y,\tilde z)$  where $\tilde z=(\omega_2(x),x_3)=(\tilde x_2,x_3)$.
Then,
$$
\nu(y,\delta_\lambda \tilde z)^2=\nu(x_1,\lambda \tilde x_2,\lambda^2 x_3,x_4)=4\lambda^{8} x_3^4+\lambda^2\tilde x_2^2,
$$
and one easily gets $\varrho_{\min}(p,\Sing)=1$, $\varrho_{\max}(p,\Sing)=4$. Thus $\varrho_{\min}(p,\Sing) < Q(p)-Q_\Sing= 3 < \varrho_{\max}(p,\Sing)$ and neither the condition of Proposition~\ref{fin} nor the one of Proposition~\ref{finito} is satisfied.
However, taking a sufficiently small neighborhood $U$ of $p$ we have
$$
\ss^7\llcorner_\Reg (U) \asymp \int_{U}\frac{1}{\nu(x)}dx \asymp \int_{[-1,1]^4} \frac{1}{\sqrt{\omega_1(x)^2+\omega_2(x)^2}}dx \asymp \int_{[-1,1]^2} \frac{1}{\sqrt{\tilde x_2^2+x_3^4}}d\tilde x_2dx_3 < \infty.
$$
Thus $\ss^7=\ss^7\llcorner_\Reg$ is a Radon measure.
\end{example}

\appendix
\section{Appendix}
%
 We show here a technical result which is needed in Section~\ref{fghjk} to provide the local equivalent of $\frac{d\ss^{Q_N}_N}{d\mu}$.
\begin{proposition}
\label{le:unifbbox}
Let $M$ be oriented, $\omega$ be a volume form on $M$, $K$ be a compact subset of $M$, and $X_1,\dots, X_{m}$ be a generating family for $(\bD,g)$. There exist a constant $C>0$ and a function $\eta: K \to (0,+\infty)$ such that the following holds for every $p \in K$.

Let $\cX=(X_{I_{1}},\dots,X_{I_{n}})$ be a $n$-tuple
in $\arg \max \{ \omega_p(X_{I'_{1}}(p),\dots,X_{I'_{n}}(p)) \mid \ X_{I'_{1}},\dots,X_{I'_{n}} \hbox{ s.t. }
\sum_i |I'_i| = Q(p) \}$. Consider another $n$-tuple of vector fields $\cY=(Y_1,\dots, Y_n)$
such that, for $i=1, \dots,n$,
$$
Y_i \in \span \{ X_{I_j} \ : \ |I_j| \leq \ell(Y_i) \}, \quad \hbox{i.e.} \quad Y_i=\sum_{|I_j| \leq \ell(Y_i)}Y_i^j X_{I_j},
$$
where all components $Y_i^j$ are smooth functions satisfying, for $s \in \N$,
\begin{equation}
\label{eq:normYs}
\| Y^{(s)} (p) \| \quad \hbox{and} \quad \| (Y^{(s)})^{-1} (p) \| \leq 2,
\end{equation}
$Y^{(s)} (p)$ being the matrix
$$
Y^{(s)} (p)=
\left\{
\begin{array}{ll}
\left( Y_i^j(p) \right)_{\{i,j \ \mid \ \ell(Y_i)=|I_j|=s\} } & \hbox{if }\{i,j \ \mid \ \ell(Y_i)=|I_j|=s\}\neq\emptyset\\
\ 1 & \hbox{otherwise.}
\end{array} \right.
$$

Then, for any $\eps \leq \eta(p)$ there holds:
\begin{eqnarray}
\label{eq:bboxY}
 \mathrm{Box}_\cY^i(p, \eps/C) \subset B(p,\eps) \subset \mathrm{Box}_\cY^i(p,C\eps), \quad i=1,2,
\end{eqnarray}
where
\begin{eqnarray*}
 & &  \mathrm{Box}_\cY^1(p,\eps) = \left\{ \exp\left(\sum_i y_i Y_i\right)(p) \ \mid \ |y_i| \leq \eps^{\ell(Y_i)}, \ i=1, \dots, n \right\}, \\
& &    \mathrm{Box}_\cY^2(p,\eps) = \{ \exp(z_1 Y_1) \circ \cdots \circ \exp(z_n Y_n)(p) \ \mid \ |z_i| \leq \eps^{\ell(Y_i)}, \  i=1, \dots, n \}.
\end{eqnarray*}
\end{proposition}

\begin{remark}
Notice that any $n$-tuples $\cX$ and $\cY$ satisfying the hypothesis of the proposition at a point $p$ form a basis adapted to the
flag \r{flagd} of the distribution at $p$. Also, $\cY=\cX$ satisfies the hypothesis and in this case the inclusions~\eqref{eq:bboxY}  for $i=2$ result from the uniform Ball-Box Theorem, \cite{jean1} and \cite[Theorem~4.7]{Jea-2014}.
\end{remark}

Let us introduce some notations.  Fix $p \in M$. Recall that the weights $w_i=w_i(p)$, $i=1,\dots,n$, are defined by setting $w_i=s$
if $\dim\bD^{s-1}_p < i \leq \dim\bD^{s}_p$.  The largest of these integers, i.e. $w_n(p)$, is an upper semi-continuous function of $p$, and hence it admits a maximum $\bar w \in \mathbb{N}$ on the compact set $K$. Thus $w_i(p) \leq \bar w$ for every $p\in K$ and $i=1,\dots,n$. Using these weights $w_i$ we also define the pseudo-norm $\| \cdot \|_p$ on $\R^n$ by
$$
 \|x\|_p = \max ( |x_1|^{1/w_1}, \dots ,|x_n|^{1/w_n}).
$$
Notice that any $n$-tuple of vector fields $\cY=(Y_1,\dots, Y_n)$ verifying the hypothesis of Proposition~\ref{le:unifbbox} satisfies, up to reordering, $\ell(Y_i) = w_i$, for every $i$. We will always suppose in this section that $\cY$ and $\cX$ have been ordered in that way. In particular $| I_i| = w_i$, for every $i$.

Let $x=(x_1,\dots,x_n)$  be a system of privileged coordinates at $p$ and let us denote by $\widehat{d}_p^x$  the distance of the nilpotent approximation at $p$ defined by means of the coordinates $x$. Then
\begin{equation}
\label{eq:d-hatd}
    d(p,q)= \widehat{d}_p^x(0,x(q)) + o(\widehat{d}_p^x(0,x(q))),
\end{equation}
where $x(q)$ are the coordinates of the point $q$ (see for instance \cite{bellaiche}).

Fix now a $n$-tuple $\cY=(Y_1,\dots, Y_n)$ verifying the hypothesis of Proposition~\ref{le:unifbbox} and denote by
$y$ and $z$  the coordinates associated with $\mathrm{Box}_\cY^1$ and $\mathrm{Box}_\cY^2$ respectively: $y=y(q)$ are the coordinates of the point $q$ satisfying
$$
q=\exp\left(\sum_i y_i Y_i\right)(p),
$$
and $z=z(q)$ are the coordinates of the point $q$ satisfying
$$
q=\exp(z_1 Y_1) \circ \cdots \circ \exp(z_n Y_n)(p).
$$
The inclusions~\eqref{eq:bboxY} are equivalent to
\begin{equation*}
\frac{1}{C} \|y(q)\|_p \leq d(p,q) \leq C \|y(q)\|_p \quad \hbox{and} \quad \frac{1}{C} \|z(q)\|_p \leq d(p,q) \leq C \|z(q)\|_p \quad \hbox{for $d(p,q)$  small enough}.
\end{equation*}
Since the coordinates $y$ and $z$ are privileged coordinates at $p$, using~\eqref{eq:d-hatd} we get that the inclusions~\eqref{eq:bboxY} are also equivalent to
\begin{equation}
\label{eq:hd_vs_psnorm}
\frac{1}{C} \|y\|_p \leq \widehat{d}_p^y(0,y) \leq C \|y\|_p \quad \hbox{and} \quad \frac{1}{C} \|z\|_p \leq \widehat{d}_p^z(0,z) \leq C \|z\|_p  \quad \hbox{for any } y,z \in \R^n.
\end{equation}

\begin{lemma}
\label{le:change_coor}
Let $x=(x_1,\dots,x_n)$ and $y=(y_1,\dots, y_n)$ be two systems of privileged coordinates at $p$. Assume that the change of coordinates formulas are
$$
x_i = \psi_i(y), \quad i=1,\dots,n,
$$
and let $\widehat{\psi}_i(y)$ be the homogeneous part of $\psi_i(y)$ of degree $w_i$ (that is, the sum of monomials of weighted degree $w_i$ in the Taylor expansion of $\psi_i(y)$). Then
%
$$
\widehat{d}_p^y (0,y) = \widehat{d}_p^x (0, \widehat{\psi}(y)),
$$
where $\widehat{\psi}(y) = (\widehat{\psi}_1(y), \dots, \widehat{\psi}_n(y))$.
\end{lemma}

\begin{proof}
This lemma may be seen as a consequence of~\cite[Proposition~5.20]{bellaiche}. We give here a direct proof.
%
Let us use the system of coordinates $x$ to identify a neighborhood of $p$ in $M$ with a neighborhood of $0$ in $\R^n$. We have
$$
\widehat{d}_p^x (0, x) = \lim_{s\to 0^+} \frac{1}{s} d(0,\delta_s x) \quad \hbox{and} \quad
\widehat{d}_p^y (0, y) = \lim_{s\to 0^+} \frac{1}{s} d(0,\psi(\delta_s y)),
$$
where as usual $\delta_s x =(s^{w_1} x_1, \dots, s^{w_n} x_n)$.

Hence $\widehat{d}_p^y (0, y) = \lim_{s\to 0^+} d(0,\delta_{1/s}\psi(\delta_s y))$. The conclusion follows from $\lim_{s\to 0^+} \delta_{1/s}\psi(\delta_s y) = \widehat{\psi}(y)$.
\end{proof}

\begin{proof}[Proof of Proposition~\ref{le:unifbbox}]
We already know that \eqref{eq:bboxY} holds with $\cY=\cX$ and $i=2$.
Let $\cY$ be a $n$-tuple of vector fields satisfying the hypothesis of Proposition~\ref{le:unifbbox}. We will prove that \eqref{eq:bboxY} holds for $\cY$ with $i=1$, the proof with $i=2$ being similar.

Let $x$ and $x'$ be the coordinates associated with $\mathrm{Box}_\cY^1$ and $\mathrm{Box}_\cX^2$ respectively, and assume that the change of coordinate formulas are $x'=\psi(x)$. From \eqref{eq:hd_vs_psnorm} and Lemma~\ref{le:change_coor}, it suffices to show that there exists a constant $C'>0$, independent of $p$ and of $\cX$, such that
\begin{equation}
\label{eq:psnormes}
\frac{1}{C'} \|x\|_p \leq \|\widehat{\psi}(x)\|_p \leq C' \|x\|_p \quad \hbox{for any } x \in \R^n.
\end{equation}

By definition of $\psi(x)$, there holds
$$
\exp\left(-\sum_i x_i Y_i\right) \circ \exp(\psi_1(x) X_{I_1}) \circ \cdots \circ \exp(\psi_n(x) X_{I_n})(p)=p.
$$
Remind that $Y_i=\sum_{w_j\leq w_i}Y_i^j  X_{I_j}$. Taking the homogeneous terms of weighted degree zero in the equality above, we obtain
$$
\exp\left(-\sum_{j=1}^n \left(\sum_{\{i \, : \, 
 w_i=w_j\}} x_i Y_i^j(p)\right) \widehat X_{I_j}\right) \circ \exp(\widehat \psi_1(x) \widehat X_{I_1}) \circ \cdots \circ \exp(\widehat \psi_n(x) \widehat X_{I_n})(p)=p.
$$
Set $y_j = \sum_{\{i \, : \, 
 w_i=w_j\}} x_i Y_i^j(p)$. The Campbell-Hausdorff formula allows to expand the product above as follows,
\begin{equation}
\label{eq:ch}
\sum_{j=1}^n (\widehat \psi_j(x) - y_j) \widehat X_{I_j}(p) +
\sum_{|I| \leq w_n} \sum_{\alpha,\beta} c_{\alpha,\beta}^I y_{\alpha_1} \cdots y_{\alpha_s} \widehat \psi_{\beta_1} (x) \cdots \widehat \psi_{\beta_r}(x) \widehat X_{I}(p)= 0,
\end{equation}
where the second sum is taken over all the multi-indices $\alpha = (\alpha_1, \dots,\alpha_s)$ and $\beta = (\beta_1, \dots,\beta_r)$ such that $s+r>1$ and $w(\alpha) + w(\beta) =  |I|$ (we use the notation $w(\alpha)$ for $w_{\alpha_1} + \cdots + w_{\alpha_s}$). Note that this implies that all $w_{\alpha_i} $ and $w_{\beta_i}$ are smaller than $|I|$. The structural constants $c_{\alpha,\beta}^I$ do not depend on $p$ but only on the coefficients of the Campbell-Hausdorff series. More precisely, there exists a constant $C_{\bar w}$ depending only on the maximum value $\bar w$ of $w_n$ on $K$ such that $|c_{\alpha,\beta}^I| \leq C_{\bar w}$ for every $I, \alpha, \beta$ appearing in~\eqref{eq:ch} at any point $p\in K$.

Now, for every multi-index $I$ there holds $\widehat X_{I}(p) = \sum_{|I_i| = |I|} X_{I}^i(p) \widehat X_{I_i}(p)$ where
$$
X_{I}^i(p) = \frac{\omega(X_{I_{1}},\dots,X_{I_{i-1}},X_I, X_{I_{i+1}},\dots, X_{I_{n}})(p)}{\omega(X_{I_{1}},\dots,X_{I_{n}})(p)}.
$$
The choice of the family $\cX=(X_{I_{1}},\dots,X_{I_{n}})$ ensures that, for every $I$ and $i$ such that $|I_i| = |I|$ we have $|X_{I}^i(p)| \leq 1$.

Plugging the expression of $\widehat X_{I}(p)$ into~\eqref{eq:ch}, we obtain
\begin{equation}
\label{eq:ch2}
\sum_{i=1}^n \left( \widehat \psi_i(x) - y_i+
\sum_{|I|=w_i} \sum_{\alpha,\beta} c_{\alpha,\beta}^I X_{I}^i(p) y_{\alpha_1} \cdots y_{\alpha_s} \widehat \psi_{\beta_1} (x) \cdots \widehat \psi_{\beta_r}(x) \right) \widehat X_{I_i}(p) = 0.
\end{equation}
Note that these equalities are in triangular form: all the indices $\alpha_j$ and $\beta_j$ appearing in the $i$th equality are smaller than $i$. Since moreover all the coefficients in the above equality are bounded independently of $p$, we get the following change of coordinates formulas, for $i=1,\dots,n$,
$$
\widehat \psi_i(x) = \sum_{\alpha \ \mathrm{s.t.} \ w(\alpha)=w_i} a'_\alpha (p) y_{\alpha_1} \cdots y_{\alpha_s} = \sum_{\alpha \ \mathrm{s.t.} \ w(\alpha)=w_i} a_\alpha (p) x_{\alpha_1} \cdots x_{\alpha_s},
$$
where $a_\alpha (p) = \sum_{w(\beta)=w_i} a'_\beta (p) Y_{\beta_1}^{\alpha_1}(p) \cdots Y_{\beta_s}^{\alpha_s}(p)$. The coefficients $a'_\alpha (p)$ are bounded independently of $p$, therefore it follows from hypothesis~\eqref{eq:normYs} that it is also the case for $a_\alpha (p)$: there exists a constant $C'_{\bar w}$ depending only on $\bar w$ such that $|a_\alpha (p)| \leq C'_{\bar w}$ for any $p\in K$ and any $\alpha$ such that $w(\alpha) \leq \bar w$.

Let $N$ be the number of multi-indices $\alpha$ such that $w(\alpha) \leq \bar w$ and set $C'=n(NC'_{\bar w})^{1/\bar w}$. There holds
$$
\sup \{ \| \widehat \psi(x) \|_p  \ | \ p \in K, \ \|x\|_p = 1 \} \leq C',
$$
and since every $\widehat \psi_i(x)$ is a  weighted homogeneous polynomial of degree $w_i$, we obtain
\begin{equation}
\label{eq:right}
\| \widehat \psi(x) \|_p \leq C' \|x\|_p.
\end{equation}

Now observe that hypothesis~\eqref{eq:normYs} implies that the vectors  $\widehat Y_i(p)=\sum_{w_j= w_i}Y_i^j(p) \widehat X_{I_j}(p)$, $i=1, \dots, n$, form a basis of $T_pM$ and that the change-of-basis matrix $M(p)$ between $(\widehat X_{I_1}(p),\dots,\widehat X_{I_n}(p))$ and $(\widehat Y_1(p),\dots,\widehat Y_n(p))$ has coefficients $M_{ij}(p)$ that are bounded independently of $p$. Identity~\eqref{eq:ch2} can then be rewritten as
$$
\sum_{i=1}^n \left( \sum_{w_j=w_i} M_{ji}(p) \widehat \psi_j(x) - x_i +
\sum_{w_j=w_i} \sum_{|I|=w_i} \sum_{\alpha,\beta} c_{\alpha,\beta}^I M_{ji}(p) X_{I}^j(p) x_{\alpha_1} \cdots x_{\alpha_s} \widehat \psi_{\beta_1} (x) \cdots \widehat \psi_{\beta_r}(x) \right) \widehat Y_i(p) = 0.
$$
Using the above reasoning in which we exchange the role of $x$ and $\widehat \psi$ coordinates, we obtain, up to enlarging the constant $C'$,
$$
\| x \|_p \leq C' \|\widehat \psi(x)\|_p.
$$
This inequality together with~\eqref{eq:right} gives exactly~\eqref{eq:psnormes}, which ends the proof.
\end{proof}

%

\end{document}